\documentclass{probarticle}

\usepackage{comment}
\usepackage[normalem]{ulem}

\DeclareMathOperator{\Li}{Li} 
\newcommand{\A}{{\mathcal A}}
\renewcommand{\N}{{\mathcal N}}
\newcommand{\F}{{\mathcal F}}
\newcommand{\be}{\begin{equation}}
\newcommand{\ee}{\end{equation}}
\newcommand{\RR}{R}

\let\Im\relax
\DeclareMathOperator{\Im}{Im}
\let\Re\relax
\DeclareMathOperator{\Re}{Re}
\newcommand{\CO}{\mathcal{C}}

\definecolor{ao}{rgb}{0.0, 0.5, 0.0}

\definecolor{rr}{rgb}{1, .5, 0.0}

\newcommand{\old}[1]{}

\begin{document}

\title{Limit shapes for the asymmetric five vertex model}
\author{Jan de Gier\footnote{University of Melbourne, Melbourne, Australia}, Richard Kenyon\footnote{Yale University, New Haven, CT}, Samuel S. Watson\footnote{Brown University, Providence, RI}}
\date{}
\maketitle
\begin{abstract}
  We compute the free energy and surface tension function for the
  \emph{five-vertex model}, a model of non-intersecting monotone lattice
  paths on the grid in which each corner gets a weight $r>0$.  We
  give a variational principle for limit shapes in this setting, and
  show that the resulting Euler-Lagrange equation can be integrated,
  giving limit shapes explicitly parameterized by analytic functions.
\end{abstract}

\section{Introduction} \label{sec:introduction}

A \textbf{six-vertex configuration} on $\Z^2$ is an orientation of the edges of $\Z^2$ with the property (called the \textit{ice rule}) that each vertex has two outgoing edges and two incoming edges, as shown in Figure~\ref{fig:exampleconfiguration}. If we draw the subset of north- or west-going edges in such a configuration, then we obtain a figure in which each vertex is one of the types in Figure~\ref{fig:sixvertices}. This subset of edges comprises a collection of edge-disjoint, non-crossing northwest-going lattice paths.

\begin{figure}[h!]
  \begin{minipage}[b]{0.495\textwidth}
    \centering
    \includegraphics[width=2.7in]{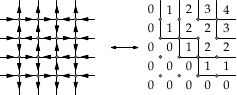}
    \captionof{figure}{A six-vertex configuration and heights}
    \label{fig:exampleconfiguration}
  \end{minipage} \hfill
  \begin{minipage}[b]{0.495\textwidth}
    \centering
    \includegraphics{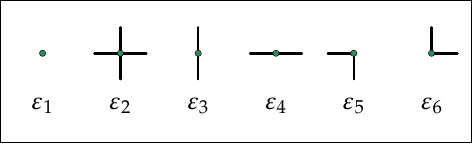}
    \captionof{figure}{The six possible configurations at each vertex} \label{fig:sixvertices}
  \end{minipage}
\end{figure}

We associate an integer-valued function $h$, called the \textbf{height function}, to a six-vertex configuration as follows. We define $h$ on the set of faces of $\Z^2$. If face $f_2$ is immediately above or to the right of $f_1$, then we require that $h(f_2) - h(f_1)$ is equal to $+1$ if the edge between $f_1$ and $f_2$ is oriented North or West, and $0$ otherwise. This condition defines $h$ uniquely up to an additive constant; note that the northwest lattice paths 
are the contours of $h$ (in other words, the paths separating consecutive values). 

Given real numbers $\eps_1, \ldots, \eps_6$ and a finite subgraph $G=(V,E)$ of $\Z^2$, we associate with each six-vertex configuration $\sigma$ on $G$ the Hamiltonian
\[
  H(\sigma) = \sum_{v \in V} \eps_{\sigma_v}, 
\]
where $\sigma$ is regarded as a map from $V$ to $\{1,2,\ldots,6\}$ which identifies the type of each vertex. We define the \textbf{partition function} $Z(G) = \sum_{\sigma}\e^{-H(\sigma)}$, where the sum is over all six-vertex configurations $\sigma$. The \textbf{six-vertex model} on $G$ is the probability measure $\P_G$ defined by 
\[
  \P_G(\sigma) = \frac{\e^{-H(\sigma)}}{Z(G)}. 
\]
Equivalently, we may define the \textbf{vertex weights} $a_i = \e^{-\eps_i}$ for $i=1,2,\ldots,6$, in which case the probability of a configuration is proportional to the product of its vertex weights.  

We consider the model with vertex weights
\begin{equation} \label{eq:basicweights}
  (a_1, \ldots, a_6) = (1,0,1,1,r,r),
\end{equation}
where $r$ is a positive real number. This is an $\eps_2 \to \infty$ limit of the six-vertex model. We can see in the context of Figure~\ref{fig:exampleconfiguration} that setting $a_2 = 0$ corresponds to eliminating the intersections between the northwest-going lattice paths. Therefore, when $r=1$, this model specializes to the monotone, non-intersecting lattice path (MNLP) model, or equivalently, the honeycomb dimer model. In the case $r=1$ the model is determinantal, and the partition function may be computed using the Karlin-MacGregor-Lindstr\"om-Gessel-Viennot method \cite{karlin1959coincidence,lindstrom1973vector,gessel1989determinants} or the Kasteleyn method \cite{kasteleyn1967graph}. This method is not available for other values of $r$.

We will also consider the family of probabilities measures, indexed by $r > 0$ and $(X,Y) \in \R^2$, corresponding to the vertex weights
\begin{equation} \label{eq:efieldweights} 
  (a_1, \ldots, a_6) = \left(1,0,\e^X, \e^Y, r\e^{\frac{X+Y}{2}}r, \e^{\frac{X+Y}{2}}\right). 
\end{equation}
Physically, we say that these new weights are the result of applying an ``electric field'' $(X,Y)$ to the original weights \eqref{eq:basicweights}; each vertical edge carries an extra weight $\e^X$ and each horizontal edge an extra weight $\e^Y$. This three-parameter family is sufficient to describe an arbitrary five-vertex model: given weights $(a_1, 0, a_3, a_4, a_5, a_6)$, we may normalize to set $a_1=1$ and replace $a_5$ and $a_6$ with their geometric mean. 
For the boundary conditions we consider (see below), the corresponding probability measure is unchanged  
by replacing $a_5$ and $a_6$ with their geometric mean, since the difference between 
the number of left turns and right turns is constant. (See however the discussion in the last paragraph of Section \ref{varprinsection} below.)

Our goal is to study the limit shape of the height function $h$, for Dirichlet boundary conditions on $G$,
see for example Figure \ref{bppr} and Figure \ref{bppbigrsim2}.
\begin{figure}[htbp]
\center{\includegraphics[width=5in]{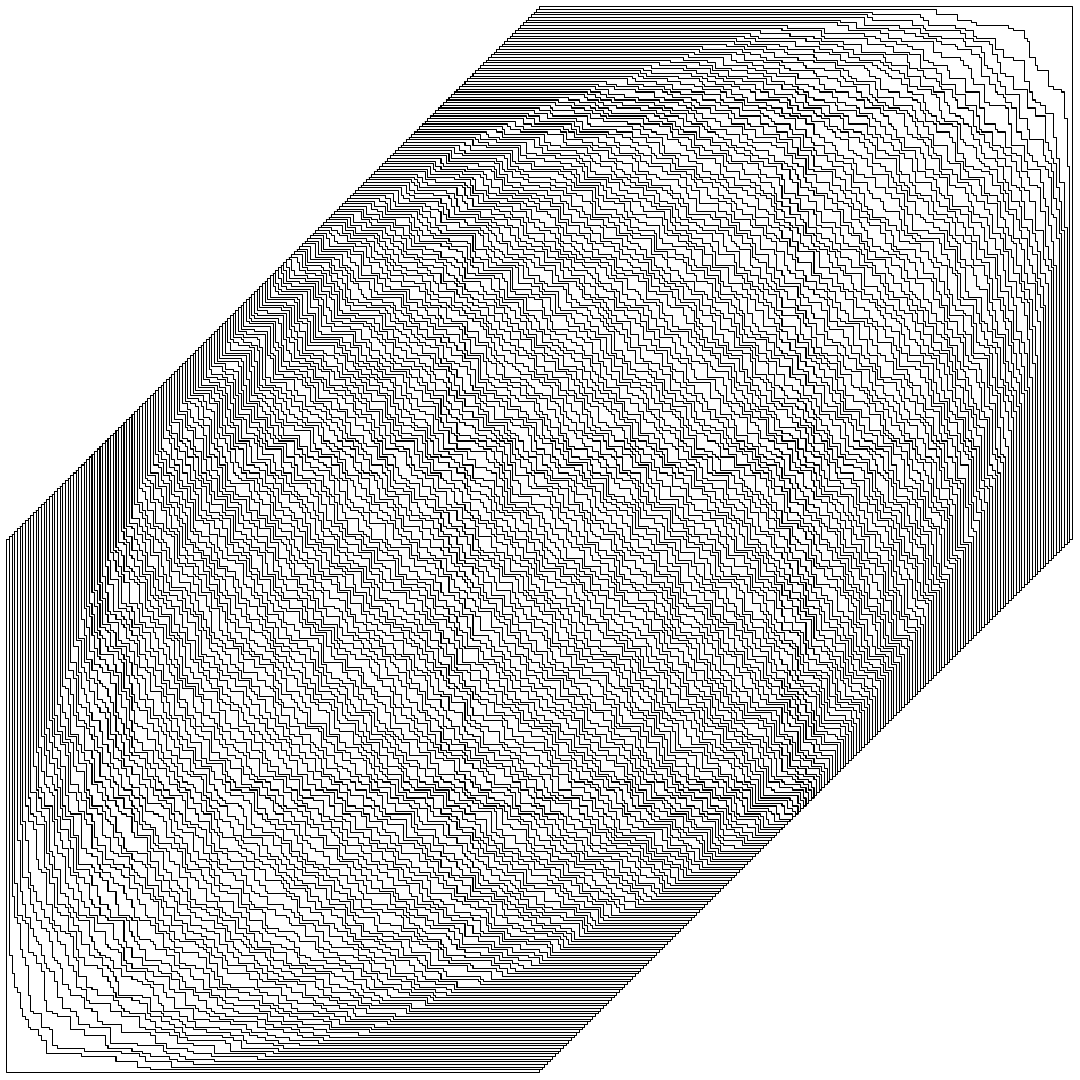}}
\caption{\label{bppr}The five-vertex model with ``boxed plane
  partition" boundary conditions and $r<1$. The lattice paths start along the
  diagonal, southeast side on the hexagon and exit along the northwest
  side. In terms of the height function, these are Dirichlet boundary
  conditions: $h$ is zero along the lower and left boundaries,
  increasing linearly along the diagonal boundaries, and constant
  (equal to the side length of the hexagon) along the upper and right
  boundaries. Shown is the case $r=0.6$ and $n=200$. Compare with the exact limit shape of Figure \protect{\ref{bppexact}}.}
\end{figure}

We will do this by computing a function $\sigma_r: \N \to \R$ called the \textbf{surface tension} of the model (see \cite{kenyon2007limit} for the $r=1$ case), where
$\N$ is the triangle with vertices at $(0,0)$, $(0,1)$, and $(1,0)$.
The function $\sigma$ defines a variational principle and thus a partial differential equation 
whose solution describes the limiting shape. The models \eqref{eq:basicweights} and \eqref{eq:efieldweights} are connected: the surface tension $\sigma_r$ and the free energy $F(X,Y)$ of \eqref{eq:efieldweights} are related by the \textit{Legendre transform} \cite{kenyon2007limit}: 
\[
  - \sigma_r(s,t) = \max \left\{ F(X,Y) - sX - tY \, : \, (X,Y) \in \R^2\right\}. 
\]

The six-vertex model has not been solved in full generality, but Sutherland, Yang, and Yang gave an explicit diagonalization (for finite system size) \cite{sutherland1967exact} of the ``transfer matrix'', based on earlier work by Lieb \cite{lieb1967exact}. It remains an important open problem to compute the asymptotics of the corresponding leading eigenvalue. The solution by Lieb, and indeed,
applications of the Bethe Ansatz in many other settings rely on nonrigorous arguments at certain points.
However the five-vertex model has a feature (an explicit form for the Bethe roots) that allows us to give a completely rigorous Bethe Ansatz argument.

The so-called free-fermionic case of the six-vertex model is the case $a_1a_2 - a_3 a_4+a_5a_6 = 0$. In this case the model is determinantal: edge probabilities are computed via determinants, and the limit shape problem admits an exact solution. This was carried out in a sequence of papers \cite{cohn2001variational,kenyon2006dimers,kenyon2006planar,kenyon2007limit}. 
The underlying PDE in this case is 
equivalent to the \textbf{complex Burgers equation} $\phi_x + \phi \phi_y = 0$, and can be solved 
using the method of
(complex) characteristics, see \cite{kenyon2007limit} and \cite{Griffithsetal}. 
This equation is also fundamental in random matrix theory and free probability.
The equation analogous to $\phi_x + \phi \phi_y = 0$ in the context of the five-vertex model 
(the model with weights \eqref{eq:basicweights}) is a generalization of the complex 
Burgers equation: the equation (\ref{BB0}). This equation is not amenable to the complex characteristics method. Nonetheless we will show how to find explicit solutions, parametrized by arbitrary analytic functions. 

When $r<1$ the surface tension $\sigma_r(s,t)$ has the interesting feature of not being strictly convex on all of $\N$; 
it is only strictly convex on a region $\N^*\subset\N$ bounded by a certain hyperbola, see Figure \ref{hyperb}.  On the complementary region $\N\setminus\N^*$ the surface tension is \emph{linear}. For the limit shape
problem we consider here, this leads to pieces of the limit shape (called ``neutral regions'') 
where the variational formalism gives no information. 
We conjecture that there is indeed \emph{no limit shape} in these regions, 
and the height function remains random in the scaling limit. The part of the limit shape corresponding to slopes where $\sigma$ is strictly convex is called the ``repulsive'' region. See Figure \ref{bppexact} for an example.

In \cite{RS} is considered a setting in which, in the neutral region, the height
satisfies a degenerate Euler-Lagrange equation, given by a variant of the Burgers' equation
(this posits that the height derivatives are constant along certain characteristic lines). A short calculation shows that this does not generally hold for the examples given in the current paper.

There have been a number of previous works on the five vertex model, see
e.g. \cite{Bogolyubov, GwaSpohn, Wuetal}. However, to our knowledge
none achieve an exact form for the free energy, nor discuss the limit
shape problem. The Bethe equations we use for the five-vertex model
are very similar to those for the asymmetric exclusion process, see
e.g. \cite{GwaSpohn, golinelli2006asymmetric}; in fact our calculation
is inspired by those works.

\subsection{Organization}
The paper is organized as follows.

In Section \ref{varprinsection} we discuss the variational principle for this model. 
The remaining sections deal with the explicit calculation of the surface tension and limit shapes.
The calculations for $r<1$ and $r>1$ are very similar but different in some details, so we start
with
the $r<1$ case, and redo the calculations for $r>1$ in a later section (Section \ref{bigr}).

In Section \ref{sec:bethe} we use the Bethe Ansatz to compute the eigenvalues of the transfer matrix
for fixed $n,N,r,X,Y$. (This is valid for both $r<1$ and $r>1$.)

In Section \ref{Max} we compute (for $r<1$) the asymptotics of the leading eigenvalue $\Lambda=\Lambda(r,n,N,X,Y)$.
The rescaled limit, as $n,N\to\infty$ with $n/N\to s$, gives the 
``microcanonical'' free energy $F_m(s,Y)$ defined by
\begin{equation}
\label{microc}F_m(s,Y):=\lim_{N\to\infty}\frac1N\log\Lambda - Xs.\end{equation}

In Section \ref{sec:legendre} the (grandcanonical) free energy $F(X,Y)$ is obtained by maximizing over all $s$:
\begin{equation}
\label{grandc}F(X,Y) = \max_{s\in[0,1]}\{F_m(s,Y)+Xs\}.\end{equation}
We also compute the surface tension $\sigma(s,t)$, its Legendre dual.

In Section \ref{sec:EL} we write down the Euler-Lagrange equation for the surface tension minimizing function
and reduce it to a first-order linear 
PDE. This PDE is solved in Section \ref{sec:limitshapes}, where we give the general solution parameterized by an
arbitrary analytic function, and give several examples.

In Section \ref{bigr} we redo the above calculations in the case $r>1$. 
In Section \ref{perron} we prove that our computation gives the leading eigenvalue.
\medskip

\noindent{\bf A note on notation.} The calculations are complicated by an abundance of variables. We review their
definitions and dependencies here. 
\begin{enumerate}
\item $r$, a positive, real constant, is the weight per corner of the lattice paths.
\item $X,Y$ are the ``external field": $e^X,e^Y$ are the weight per vertical, respectively horizontal, edge.
\item $s,t$ are slope variables: $s$ is the density of lattice paths per horizontal lattice spacing, and $t$ is the density of lattice paths per vertical lattice spacing. 
They are also the $x$- and $y$- derivatives of the expected height function. 
\item $\N$ is the set of allowed slopes: $\N$ is the convex hull of $(0,0),(1,0),(0,1)$.
\item $\N^*$ is the maximal subset of $\N$ on which $\sigma$ is strictly convex; see Figure \ref{hyperb}.
\item $N$ is the circumference of the cylinder.
\item $n$ is the number of particles on the cylinder. We take the limit $n,N\to\infty$ with $n/N\to s$.
\item $w_0$, also called $w$ in Sections \ref{sec:legendre} and onward. It is 
a complex number in the upper half plane; it is determined by and determines $s$ and $Y$
(See Figures \ref{twoanglefoliation} and \ref{wztriangles}). It serves as a ``conformal coordinate'' for the model
(we could equally well use $u_0$ or $\bar z_0$ as conformal coordinates).
\item For $r<1$, $u_0 = 1-(1-r^2)\bar w_0$. Also called $u$ in Sections \ref{sec:legendre} and onward. See Figure \ref{trianglefig}. For $r>1$, $u_0=1+(r^2-1)w_0$. In either case $R=|u_0|$.
\item $z_0$, also called $z$ in Sections \ref{sec:legendre} and onward. Its complex conjugate 
$\bar z_0$ is the analog of $w_0$ 
under the symmetry interchanging $s$ and $t$; $z_0$ satisfies $1-w-z+(1-r^2)zw=0,$ see Lemma \ref{spcurvesmall}.
This holds for both $r<1$ and $r>1$. 
\item $u^*$ is defined by $u^*\bar u=r^2$. 
\item For $r<1$, $\theta$ is defined so that the argument of $w_0$ is $s\theta$; in other words, $\theta = \arg\left(\frac{w_0}{1-w_0}\right)$. See Figure \ref{wztriangles}. For $r>1$, $\pi-s\theta=\arg w_0$. See Figure \ref{wbigr}.
\end{enumerate}
\bigskip

\noindent{\bf Acknowledgements.} Research of RK is supported by NSF grants DMS-1612668, DMS-1713033
and the Simons Foundation award 327929. JdG is support by the Australian Research Council through the ARC Centre of Excellence for Mathematical and Statistial Frontiers (ACEMS). We thank Amol Aggarwal, Hugo Duminil-Copin, Vadim Gorin,
Andrei Okounkov, and Istv\'an Prause for conversations and ideas about this project, and the referee for a very careful report on the paper.

\section{The variational principle}\label{varprinsection}

Let $U\subset\R^2$ be a piecewise smooth Jordan domain, and $u:\partial U\to\R$ continuous.
We say $u$ is \emph{feasible} if there exists a Lipschitz function $h:U\to\R$ with boundary values
$h|_{\partial U} = u$ and gradient in $\N$, the convex hull of $(0,0),(1,0),(0,1)$. We let $\Omega=\Omega(U,u)$ be the set of such 
$\N$-Lipschitz extensions of $u$.

For $\eps>0$ let $\eps\Z^2$ be the graph $\Z^2$ scaled by $\eps$. Let $U_\eps\subset\eps\Z^2$ be the 
(largest connected) subgraph of $\eps\Z^2$ whose vertices are $U\cap\eps\Z^2$ and edges connect nearest neighbors. Feasibility of $u$ guarantees the existence,
for small enough $\eps$, of a five-vertex configuration on $U_\eps$ whose associated
height function $h:U_\eps^*\to\Z$ has scaled boundary values $\eps h$ 
uniformly approximating $u$ (here $U_\eps^*$ is the dual graph). 
We let $\Omega_\eps=\Omega_\eps(U,u,h)$ be the set of all height functions (of five-vertex configurations)
having the same boundary values as $h$. Note that the paths start at specified boundary vertices and along specified
directions starting from these boundary vertices: these directions are determined by the heights on the adjacent faces.

On $\Omega_\eps$ we define the probability measure $\mu_\eps=\mu_\eps(r)$ 
choosing a configuration proportional to its weight,
which is the product of its vertex weights.
(The weight is not defined at a boundary vertex, where a lattice path begins or ends; these vertices
are given weight $1$, or equivalently, we take the product over weights of \emph{interior} vertices.) Note that
the weight of a configuration does not depend on $X$ or $Y$ since the number of horizontal
and vertical edges of the paths is fixed by the boundary data. It does, however, depend on $r$.

The following statement of the variational principle is due to \cite{CKP}, who originally applied it to the domino
tiling model (and lozenge tiling model). Their proof applies identically to the five-vertex model, with several provisos.
Firstly, the strict convexity of the surface tension function $\sigma$ is essential to get uniqueness.
In the $r<1$ case here we lose strict convexity, and so get a weaker statement.
Secondly, the explicit form of $\sigma$ is different in our setting. Indeed, the explicit form of $\sigma$ is the main contribution of the present paper. Finally, the argument of \cite{CKP} relies on an a priori estimate given in \cite{CEP}
(using Azuma's inequality) for the height fluctuations for a domain with linear boundary conditions; this argument
applies to general random tiling models and in particular also to the current setting.

\begin{theorem}[\cite{CKP}, Theorem 1.1]\label{ckpthm} Take a piecewise 
smooth Jordan domain $U$ and boundary height function $u$ as above,
satisfying the property that there is a feasible extension $h$ having gradient in $\N$. 
Then for any $\delta>0$, with probability tending to $1$
a rescaled $\mu_\eps$-random height $\eps h_\eps$ will lie within $\delta$ of the unique minimizer (if $r>1$)
or one of the minimizers (if $r<1$) of the
surface tension integral
$$\iint_U \sigma(\nabla h)\,dx\,dy,$$
where $\sigma$ is given in Proposition \ref{stension} for $r<1$ and \ref{stensionbig}) for $r>1$. 
\end{theorem}

For examples, see Figure \ref{bppexact} for an $r<1$ case and Figure \ref{bppbigrsim2} for an $r>1$ case.

Even though minimizers are not typically unique in the $r<1$ case, convexity of $\sigma$ implies that
the set of all minimizers is a convex set (the average of two minimizers is also a minimizer),
and moreover all minimizers $h$ agree on the ``repulsive'' region
where $\nabla h\in\N^*$, that is, if $h_1,h_2$ are two minimizers and $\nabla h_1\in\N^*$ on a subregion $U'\subset U$ then $h_2=h_1$ on $U'$. We call $U'$ the \emph{repulsive region}. See Figures \ref{bppexact}, \ref{bppbdy} for the repulsive region for the ``boxed plane partition'' boundary conditions of Figure \ref{bppr}, and see further discussion
in Section \ref{bppsection}.

\old{
There is another natural choice of boundary conditions, which we won't consider here: rather than fix the initial
\emph{directions} of the lattice paths in $U_\eps$, one can fix just their starting and ending vertices. 
A particular path has the freedom to take its first step either N or W, for example. In this case the number of 
vertices of type 5 and 6 is not fixed by the boundary conditions, 
so one can get a larger class of measures by considering different
values for $\eps_5$ and $\eps_6$. Since this changes the weight of a configuration by 
only a linear exponential factor (linear in the boundary length), however,
it is not hard to show that in the scaling limit this choice will not affect the limit shape.}

\section{Bethe Ansatz equations} \label{sec:bethe}

\subsection{Bethe ansatz for the five-vertex model}

Given $X$, $Y$, and $r$ as described in Section~\ref{sec:introduction} and another variable $h>0$, 
following \cite{nolden1992asymmetric} we define
the six-vertex model vertex weights
\[
(a_1, \ldots, a_6) = (1,\e^{-4h+X-Y},\e^{X},\e^{Y},r\e^{\frac12(X+Y)},r\e^{\frac12(X+Y)}). 
\]
Taking $h\to\infty$ yields \eqref{eq:efieldweights}.
We also define the quantities
\begin{align*}
  2\Delta&=\frac{a_1a_2+a_3a_4-a_5a_6}{\sqrt{a_1a_2a_3a_4}}=
           \e^{2h+Y}(1-r^2) + \e^{-2h-Y}\\
  H &= \e^{2h}.
\end{align*}

\insetfiguremanual[0.7\textwidth][0.26\textwidth][t]{%
 Consider the model \eqref{eq:basicweights} on the cylinder of height $M$ and circumference $N$, as shown in Figure~\ref{fig:cylinderpaths}. The states of this model are specified by the vertical edges, since the horizontal edges are uniquely determined by the requirement---implicit in the ice rule---that they connect up the vertical ones to form
 disjoint northwest-going paths from the bottom to the top of the cylinder.
  
The $2^N \times 2^N$ \textit{transfer matrix} $T$ is indexed by configurations of edges in each row, with $T(x,y)$ defined to be the product of Boltzmann weights of the vertices in a circle around the cylinder, with the configuration $x$ in the row below and $y$ in the row above it, with the horizontal edges filled in as necessary. If $x$ and $y$ are such that no configuration of horizontal edges in the circle would comply with the ice rule, we set $T(x,y) = 0$. If $x$ and $y$ are both the empty configuration, then there are two valid configurations of horizontal edges between them, and in that case we define $T(x,y)$ to be the sum of the product of Boltzmann weights over each such configuration.%
}{\includegraphics[width=2.5cm]{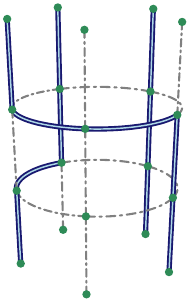}}
{
  \captionof{figure}{A five-vertex model on a cylinder with circumference $N = 5$}
  \label{fig:cylinderpaths}
}  
The transfer matrix is related in a straightforward manner to the partition function via $Z = \trace T^M$ \cite[Chapter 8]{baxter2016exactly}. 
Therefore, the $M\to\infty$ behavior of the partition function is governed by the largest eigenvalue 
$\Lambda=\Lambda_N$ of $T$. The limit $\lim_{N\to\infty} \frac1N\log\Lambda$ is the \emph{free energy}.
One could in principle get the same free energy from the simultaneous limit $M,N\to\infty$ of $\frac1{MN}\log Z$, see \cite{duminilcopin2016discontinuity} for a rigorous argument, but it is significantly easier to take first the $M\to\infty$ limit, then the $N\to\infty$ limit.

Since the number of vertical edges is constant from row to row, we may write $T$ in block diagonal form using $\binom{N}{n} \times \binom{N}{n}$ blocks $T_n = T_n(r,X,Y)$, which we define by restricting $T$ to 
configurations with $n$ vertical edges.
For fixed $n$, the $X$-dependence of $T_n$ is simply due to a factor $\e^{Xn}$, that is, $T_n$ is $\e^{Xn}$ times a matrix
independent of $X$. In the limit $N,n\to\infty$ with $n/N\to s$, we define the 
microcanonical free energy $F_m(s,Y)$ to be
$$F_m(s,Y) = -Xs+ \lim_{N\to\infty}\frac1N\log\Lambda_{n},$$
where $\Lambda_n$ is the leading eigenvalue of $T_n$.
Subtracting $Xs$ here removes the explicit dependence on $X$.
The leading eigenvalue of $T$ for fixed $N$ is obtained by maximizing over $n$ the leading eigenvalue of $T_n$. 
In the limit this corresponds to maximizing $F_m(s,Y)+sX$ over $s$, and we define the free energy 
$F(X,Y)$ as in (\ref{grandc}). 

We can calculate the $1 \times 1$ matrices $T_0$ and $T_N$ quite simply. For example, $T_0$ is a $1\times 1$ matrix with entry $a_1^N + a_4^N$, since the possibilities are that every vertex is type 1 or that every vertex is type 4. Therefore, its eigenvalue is
\[
  a_1^N + a_4^N = 1 + \e^{YN}. 
\]
If there are $N$ edges in each row, then every vertex must be of type 3. So the eigenvalue of $T_N$ is 
\[
  a_3^N = \e^{XN}. 
\]
If $n = 1$, then we can associate with each configuration $x$ the location $k$ of the occupied edge. By the rotational symmetry of the model, $T_1$ is a \textit{circulant} matrix, meaning that each row is obtained from the previous one by applying a one-entry circular shift. The eigenvalues of an $N\times N$ circulant matrix are given by \cite{davis2012circulant}
\[
\lambda_j = c_0 + c_1 \omega_j^{-1} + c_2 \omega_j^{-2} + \cdots + c_{N-1} \omega_j^{-N+1}\qquad (j = 0, \ldots, N-1),
\]
where $\omega_j = \exp(2\pi i j / N)$ is an $N$th root of unity and $(c_0, \ldots, c_{N-1})$ is the first row of the matrix. Substituting weights to find the entries, we get 
\[
\lambda_j = \e^X\left(1+r^2\left(\frac{\e^Y}{\omega_j}+\dots+\frac{\e^{(N-1)Y}}{\omega_j^{N-1}}\right)\right) = \e^{X}\left(1 + \frac{r^2\e^{Y}}{\omega_j-\e^{Y}}+\e^{YN}\frac{\omega_j r^2}{\e^{Y}-\omega_j}\right). 
\]
Now suppose $n>1$. The \textbf{Bethe ansatz} is the idea to look for eigenvectors of $T_n$ of the form
\begin{equation}  \label{eq:betheansatz} 
  f(k_1,\dots,k_n) = \sum_{\pi\in S_n} A_{\pi}\zeta_{\pi(1)}^{k_1}\dots \zeta_{\pi(n)}^{k_n},
\end{equation}
where the sum is over the permutation group $S_n$, the quantities $A_\pi, \zeta_1,\dots,\zeta_n$ are certain constants, and $1 \leq k_1 < k_2 < \ldots  < k_n \leq N$ are the positions of the occupied edges. Substituting \eqref{eq:betheansatz} into the eigenvector equation $T_nf = \lambda f$, one finds that there exist coefficients $A_\pi$ such that $f$ is indeed an eigenvector if the values $\zeta_1,\dots,\zeta_n$ satisfy the \textit{Bethe equations}, which are worked out in \cite{nolden1992asymmetric} for arbitrary $a_1,\dots,a_6$: 
\be\label{BE}
\zeta_i^N = (-1)^{n-1} \prod_{j=1}^n \frac{1+H^2 \zeta_i \zeta_j -2\Delta H \zeta_i}{1+H^2 \zeta_i \zeta_j -2\Delta H \zeta_j}. 
\ee
An underlying assumption that ensures that \eqref{eq:betheansatz} is non-zero is that the solutions $\zeta_j$ to \eqref{BE} have to be distinct.
For the general 6-vertex model, this system of coupled polynomial equations is difficult to work with.
Taking $h\rightarrow\infty$, however, the equations (\ref{BE}) yield 
\[
  \zeta_i^N = (-1)^{n-1} \prod_{j=1}^n \frac{1 -(1-r^2)\e^{Y} \zeta_j^{-1} } {1-(1-r^2 )\e^{Y}\zeta_i^{-1}},
\]
or equivalently
\begin{equation} \label{eq:BE}
\zeta_i^{N-n} ((1-r^2 )\e^{Y} - \zeta_i)^n = - \prod_{j=1}^n (1 -(1-r^2)\e^{Y} \zeta_j^{-1}) .
\end{equation}
This system has an advantage over (\ref{BE}) since the right-hand side is symmetric in all $\zeta_j$,
and the left-hand side depends only on a single root $\zeta_i$,
so the system ``decouples'': see the next section.
 
The eigenvalue corresponding to the eigenvector $f$, expressed in terms of the corresponding solution of \eqref{eq:BE}, is 
(see equation 2.6 in \cite{nolden1992asymmetric})
\be\label{eq:evalue}
\Lambda :=\e^{Xn}\left[\prod_{j=1}^n\left(1+\frac{r^2\e^{Y}}{\zeta_j-\e^{Y}} \right) + \e^{YN} \prod_{j=1}^n \frac{r^2}{\e^{Y}\zeta_j^{-1}-1}\right].
\ee

It is an interesting and difficult problem to show that all eigenvectors of $T_n$ have this form. 
However we are only concerned in this paper with the maximal eigenvalue; Section \ref{perron} below completes the proof that the above is a correct expression for the leading eigenvalue.

\subsection{Leading eigenvalue}

We assume  throughout the paper that $r \neq 1$. 
We are interested in the limit as $n,N\to\infty$ of the solution of \eqref{eq:BE} and the 
eigenvalue \eqref{eq:evalue} with edge density $n/N$ tending to a constant $0 \leq s \leq 1$. 
Define $w_j = \frac{\zeta_j}{(1-r^2)\e^Y}$. In terms of the $w_j$'s, the Bethe equations \eqref{eq:BE} become
\begin{equation} \label{eq:Adef}
  w_j^{N-n} (1-w_j)^{n} = -((1-r^2)\e^{Y})^{-N} \prod_{k=1}^n \frac{w_k-1}{w_k}.
\end{equation}
Define $\mathbf{w} \colonequals (w_1, \ldots, w_n)$.  
We define $A(\mathbf{w})$ to be the right-hand side of (\ref{eq:Adef}):
\be\label{Adef2}A(\mathbf{w}) = -((1-r^2)\e^{Y})^{-N} \prod_{k=1}^n \frac{w_k-1}{w_k}.\ee
Since $A(\mathbf{w})$ is symmetric in the $\{w_j\}$, we see that there is an equation of the form $w^{N-n}(1-w)^n = y$, where $y \in \C$, which is satisfied by $w_j$ for all $1 \leq j \leq n$. 

Curves $\CO_{a,b,c}$ of the form $a \log |w| + b \log |1-w| = c$, where $a$ and $b$ are positive, are called \textbf{Cassini ovals} (see Figure~\ref{fig:cassini}).  Note that solutions of the equation $w^{N-n}(1-w)^n = y$ lie on the Cassini oval $\CO_{a,b,c}$
with $a = N-n$, $b = n$, and $c = \log|y|$.

\begin{figure}
\centering
\includegraphics{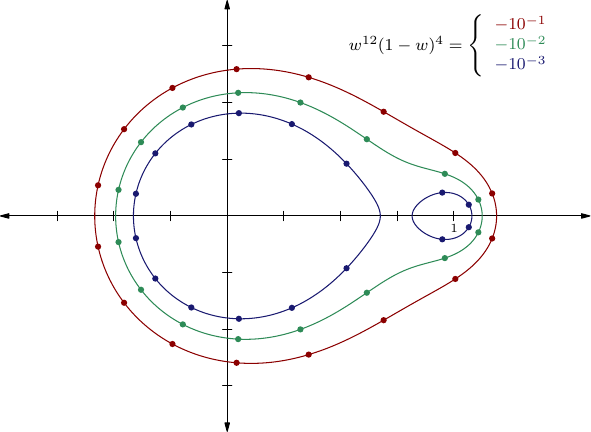}
\caption{Solutions to $|w^{12}(1-w)^{4}| = |y|$ are ``Cassini ovals'': these are the solid curves shown
for $y \in \{-10^{-1}, -10^{-2}, -10^{-3}\}$ . The dots on each curve are the solutions to the polynomials $w^{12}(1-w)^{4} = y$.}\label{fig:cassini}
\end{figure}
The eigenvalue \eqref{eq:evalue} may be expressed in terms of the $w_j$'s as 
\begin{align}\label{eq:eval1}
  \Lambda = \Lambda(\mathbf{w}) &= \e^{Xn}(1-r^2)^n\left[\prod_{j=1}^n \frac{1-w_j}{1-(1-r^2) w_j}  + \e^{YN}\prod_{j=1}^n \frac{r^2w_j}{1-(1-r^2)w_j}  \right]
\\
\label{eval2}&= \e^{Xn}(1-r^2)^n\e^{YN}\left[1-(-1)^nyr^{-2n}(1-r^2)^N  \right]
\prod_{j=1}^n \frac{r^2w_j}{1-(1-r^2)w_j}
\end{align}
where $y=A(\mathbf{w})$ is given in (\ref{Adef2}).

Since $\Lambda>0$ for the leading eigenvalue (by Perron-Frobenius), we may take the absolute
value of the right-hand side.  
\begin{equation}
\label{Lambdaabs}
\Lambda = \e^{Xn}|1-r^2|^n\e^{YN}\left|1-(-1)^nyr^{-2n}(1-r^2)^N  \right|
\prod_{j=1}^n \frac{r^2|w_j|}{|1-(1-r^2)w_j|}.\end{equation}

Our goal is to find $\Lambda$ as a function of $n,N$ and $Y$.  However
due to the implicit nature of the equations, it is convenient to
``work backwards'' and start with $n,N$ and $y$ which define the
Cassini oval. We will then reconstruct $Y$ as a function of $y$, and
\emph{a posteriori} show that, for fixed $n/N$, the map $y\to Y$ is
monotone bijective on the relevant ranges. For fixed $N,n$ and $r<1$, the quantities $Y,y$ are in fact \emph{not} monotonically related over their \emph{full} range, see Figure \ref{noseplot}. 
However for large $N$ we will show (Lemma \ref{largeY}) that the relevant part of the range will be
$Y\le-\log(1-r^2)$, where Lemma \ref{mono} shows that $Y$ and $y$ are indeed monotonically related. 

If $r>1$, then $y,Y$ are monotonically related over their full range, see discussion in Section \ref{onecpt}.

\subsection{Cassini oval properties}

Define the polynomial 
\begin{equation} \label{eq:wpoly} 
  p(w) = w^{N-n}(1-w)^{n}. 
\end{equation} 
In the following lemma we record some basic facts about the level sets of $p$ and $|p|$. 
\begin{lemma} Equations of the form $p(w) = y$ have the following properties: \label{Cassiniprops}
  \begin{enumerate}[(i)] 
 \item  When $|y|<\frac{n^n(N-n)^{N-n}}{N^N}$, the Cassini oval $|p(w)| = |y|$ consists of two components---one of which surrounds the origin and the other of which surrounds $1\in \C$---and the equation $p(w) = y$ has $N-n$ solutions on the former and $n$ solutions on the latter.
 \item When $|y|>\frac{n^n(N-n)^{N-n}}{N^N}$, the Cassini oval $|p(w)| = |y|$ is a simple closed curve surrounding both 0 and $1$, on which lie all $N$ solutions of the equation $p(w) = y$. The rightmost $n$ roots are separated from the rest by the curve
$\frac{\arg(w)}{\arg(1/(1-w))}=\frac{n}{N-n}$.
 \item Any circle centered at a point of $(-\infty,0]\cap[1,\infty)$ intersects a Cassini oval in at most two points. 
 \end{enumerate}
\end{lemma}

\begin{proof}
For the first parts of (i) and (ii), note that $|w^{N-n}(1-w)^n|$ has two local minima, at $0$ and $1$, 
and one saddle point at $w$ where $\frac{N-n}{w}-\frac{n}{1-w}=0$, that is, $w=(N-n)/N$, where it has value
$\frac{n^n(N-n)^{N-n}}{N^N}.$ The second part of (i) follows by continuity of the roots from the small $y$ case.
For the second part of (ii), note that the principal branch of the map $w\mapsto w^{(N-n)/n}(1-w)$ on $\C\setminus(-\infty,0)$ 
maps the roots of $p(w)-y$ to the circle around $0$ of radius $|y|^{1/n}$, and it maps the right-most $n$ roots to 
$y^{1/n}\e^{2\pi ij/n}$. If there is one component, the rightmost preimage of $-|y|^{1/n}$ under this map is the intersection of the oval with the curve 
$\frac{\arg(w)}{\arg(1/(1-w))}=\frac{n}{N-n}.$ 

For (iii), we first show that the oval and the circle intersect transversely
at any nonreal point $z$ of intersection. Suppose a circle is centered at $c\in\R$ and consider 
the Cassini oval $\CO_{\alpha,\beta,\gamma}$. 
As $z$ moves along  the oval, we have (differentiating $\alpha\log|z|+\beta\log|1-z|$ with respect to $z$) 
$$\Re\left[\left(\frac{\alpha}{z}-\frac{\beta}{1-z}\right)dz\right]=0.$$
At a point $z$ of non-transversal intersection, we would 
have additionally $\Re\left[\frac{dz}{z-c}\right]=0$. These two equations would imply
that $dz=0$ unless
$\frac{\alpha}{z}-\frac{\beta}{1-z}$ and $\frac{1}{z-c}$ point in the same direction,
that is, unless their ratio is real:
$$\left(\frac{\alpha}{z}-\frac{\beta}{1-z}\right)(z-c)\in\R.$$
However we claim that this quantity can be real (for $z\not\in\R$)
only if $c\in[0,1]$. To see this, for real $t$, note that the equation
$(\frac{\alpha}{z}-\frac{\beta}{1-z})(z-c)-t=0$ has two roots
$z_1,z_2$.  The discriminant of this quadratic is
$$-4 c \alpha (-t + \alpha + \beta) + (t - \alpha - c \alpha -c \beta)^2$$ which is minimized at $t=\alpha - \alpha c + \beta c$ and takes value $4 \alpha\beta (-1 + c) c$ there,
which is positive for $c\not\in(0,1).$ Thus $z_1,z_2$ are real for $c\not\in(0,1)$.

Thus the oval and the circle intersect transversely
at any nonreal point $z$ of intersection. Fix the circle center $c>1$ and increase its 
radius $t$ continuously from zero (the $c<0$ case is analogous).
By the above, the number of intersection points with the oval only 
changes when the circle passes one of the \emph{real} points of the oval.
The increasing circle will first intersect the right-most point of the oval, which is closer to $c$ than the other real points. If the oval has one component, the number of intersections only changes again when the circle passes through the left-most point of the oval, where the number of intersections drops to zero (if it went to up to four or more, increasing the radius further would lead to a tangency, contradicting the transversality).  
If the oval has two components, the next intersection of the circle is with the left real point of the right component,
and the number of intersections with the circle then drops to zero, until the circle grows to hit the left component
(at its right-most point). In all cases there are at most two intersections. 
\end{proof}

\subsection{Consistency equations}

Let's consider the consistency equations $p(w_j) = A(\mathbf{w})$, $1 \leq j \leq n$. If $\mathbf{w}$ is a solution to this system, then
\[
1=\left|\frac{A(\mathbf{w})}{p(w_j)}\right| = \frac{\left|\frac{1}{(1-r^2)\e^{Y}}\right|^N\displaystyle{\prod_{j=1}^n\frac{|w_j-1|}{|w_j|}}}{\displaystyle{\prod_{j=1}^n|w_j|^{(N-n)/n}|1-w_j|}}=\left(\frac{1}{|1-r^2|\e^{Y}}\right)^N\prod_{j=1}^n|w_j|^{-N/n}. 
\]
Taking the logarithm of both sides and dividing by $N$ gives the equation
\begin{equation} \label{eq:cdef} 
  0 = -\log(|1-r^2|\e^{Y})  - \frac{1}{n}\sum_{j=1}^n \log|w_j|.
\end{equation}
Then \eqref{eq:cdef} can be substituted into \eqref{Lambdaabs} to yield 
\begin{equation}
  \label{eq:eval3}
  \Lambda(\mathbf{w}) = \e^{Xn}\e^{Y(N-n)}\left|1-(-1)^nA(\mathbf{w})r^{-2n}(1-r^2)^N  \right|
  \prod_{j=1}^n \frac{r^2}{|1-(1-r^2)w_j|}.
\end{equation}
We can see that, for fixed $y=A(\mathbf{w})$, to maximize $\Lambda$ we must minimize the moduli of $1-(1-r^2)w_j$. 
If $r< 1$, this means (by part (iii) of the lemma) that $\Lambda$ is maximized when $w_1, \ldots w_n$ are distinct roots of $p(w) = y$ with  maximal real part. 
When $r > 1$, $\Lambda$ is maximized when $w_1, \ldots w_n$ are distinct roots of $p(w) = y$ with \textit{minimal} real part. We treat these two cases separately. 

\begin{lemma}
 \label{lem:y<0}
 Suppose the Cassini oval has one component. Then (for the maximal eigenvalue) $y<0$.
\end{lemma}

\begin{proof}
We suppose $r<1$. The case $r>1$ is symmetric (changing `largest' to `smallest').
By the uniqueness of the maximal eigenvalue, we must have $y\in\R$: otherwise
taking the $n$ largest roots $w_j$ for $y$ and $\bar y$ would give us two distinct complex conjugate 
leading eigenvectors.
The roots $w_j$ satisfy
$w^{(N-n)/n}(1-w)=y^{1/n}\sigma$, where $\sigma$ is an $n$th root of $1$. 
If $y>0$ and $n$ is even, the root $w$ 
of $p(w)=y$ with largest real part is itself real, and the subsequent roots $w_j$ of smaller real parts come in complex conjugate pairs; thus the $2j$th and $2j+1$-st roots with largest real part have the same modulus, but we can't pick an even number of these with largest modulus. This contradicts the unicity of the maximum eigenvalue.  
Similarly if $y>0$ and $n$ is odd,
the largest roots are nonreal, and come in complex conjugate pairs. We can't pick an odd number of these of maximal real part, again contradicting the unicity of the maximum eigenvalue.
If $y<0$, however, for $n$ even the roots $w$ with largest real part all come in complex conjugate pairs (as in Figure \ref{fig:cassini}), and for $n$ odd the largest root is real and the remaining ones come in complex conjugate pairs, so there is
a unique choice of $n$ largest roots.
\end{proof}

\section{Maximal eigenvalue: the $r<1$ case} \label{Max}

\subsection{Large $Y$ case}

Let us first consider the case of large $Y$, without fixing $n$.

\begin{lemma}\label{frozen} Suppose that $Y>X$ and $Y\ge -\log(1-r^2)$. Then the
measure on configurations in the limit $N\to\infty$ is supported on a single configuration consisting
of all horizontal edges. Likewise if $X>Y$ and $X\ge -\log(1-r^2)$, in the $N\to\infty$ limit
the system consists in a single
configuration of all vertical edges.
\end{lemma}

\begin{proof} Consider the model on a finite cylinder of circumference $N$ and height $M$, with free boundary conditions.
Suppose some horizontal edge is not in the configuration. The vertex to its right either 
has an edge north out of it or does not have an edge east out of it (or both of these conditions happen). Likewise from a vertical edge in the configuration, its upper vertex has an edge north out of it or no edge east out of it. Starting from a missing horizontal edge on the bottom row of the cylinder we can make a NE lattice path of ``defects" (referred to as steps below) consisting of these missing horizontal edges and present vertical edges, where we never take two consecutive vertical steps.  A general 5-vertex configuration can be considered as a  disturbance of the configuration with all horizontal edges by a family of such edge-disjoint defect paths.

Each such defect path is a concatenation of a sequence of subpaths consisting of either a horizontal step or a vertical step followed by a horizontal step. Compared to the underlying weight of the all-horizontal configuration which it displaces, a horizontal step has weight $e^{-Y}$ and a vertical-followed-by-horizontal
step has weight $r^2e^{X-Y}$. The effective relative weight of all defect paths starting from a horizontal 
non-edge 
is then bounded above by $(e^{-Y}+r^2e^{X-Y})^K=\lambda^K,$ where $K$ is the number of horizontal steps, and $\lambda<1$ because of the assumptions that $Y>X$ and $Y\ge -\log(1-r^2)$.
Note that starting from a horizontal non-edge on the bottom row, a defect path has length $K\ge M$. 

Thus each such defect path will be exponentially suppressed:
replacing each path with a path of non-defects increases the weight of a configuration by a quantity exponential in $M$. In other words, for $Y \ge -\log(1-r^2)$ the entropic contribution of there being more defective configurations than the single all-horizontal configuration can be ignored and the non-defective configuration has exponentially greater weight than the sum of weights of all possible defective configurations having a given defect edge. 

If a configuration has more than one defect path, we must take into account the change in weight when two defect paths share one or more vertices. At each vertex of intersection where two defect paths meet, the  
local contribution $e^{-Y}+r^2e^{X-Y}$ from the upper path is replaced by the single quantity
$e^{X-Y}$. Note that $e^{X-Y}<1$; we can combine this factor with the factor $e^{-Y}+r^2e^{X-Y}$ from the lower path, giving the configuration a local combined contribution of $(e^{X-Y})(e^{-Y}+r^2e^{X-Y})<\lambda$ to the relative weight. 
In particular giving relative weight $\lambda^{1/2}$ instead of $\lambda$ to each horizontal step of each defect path, we can still upper
bound the relative weight of multiple paths by the product of $\lambda^{K/2}$ over each path.

In conclusion, the \emph{union} of all configurations with one or more defect paths has exponentially (in $M$) small weight (and hence probability) compared with the weight of the all-horizontal configuration. 
\end{proof}

\subsection{Cassini oval two component case}

\begin{figure}[htbp]
  \centering
  \includegraphics[width=14cm]{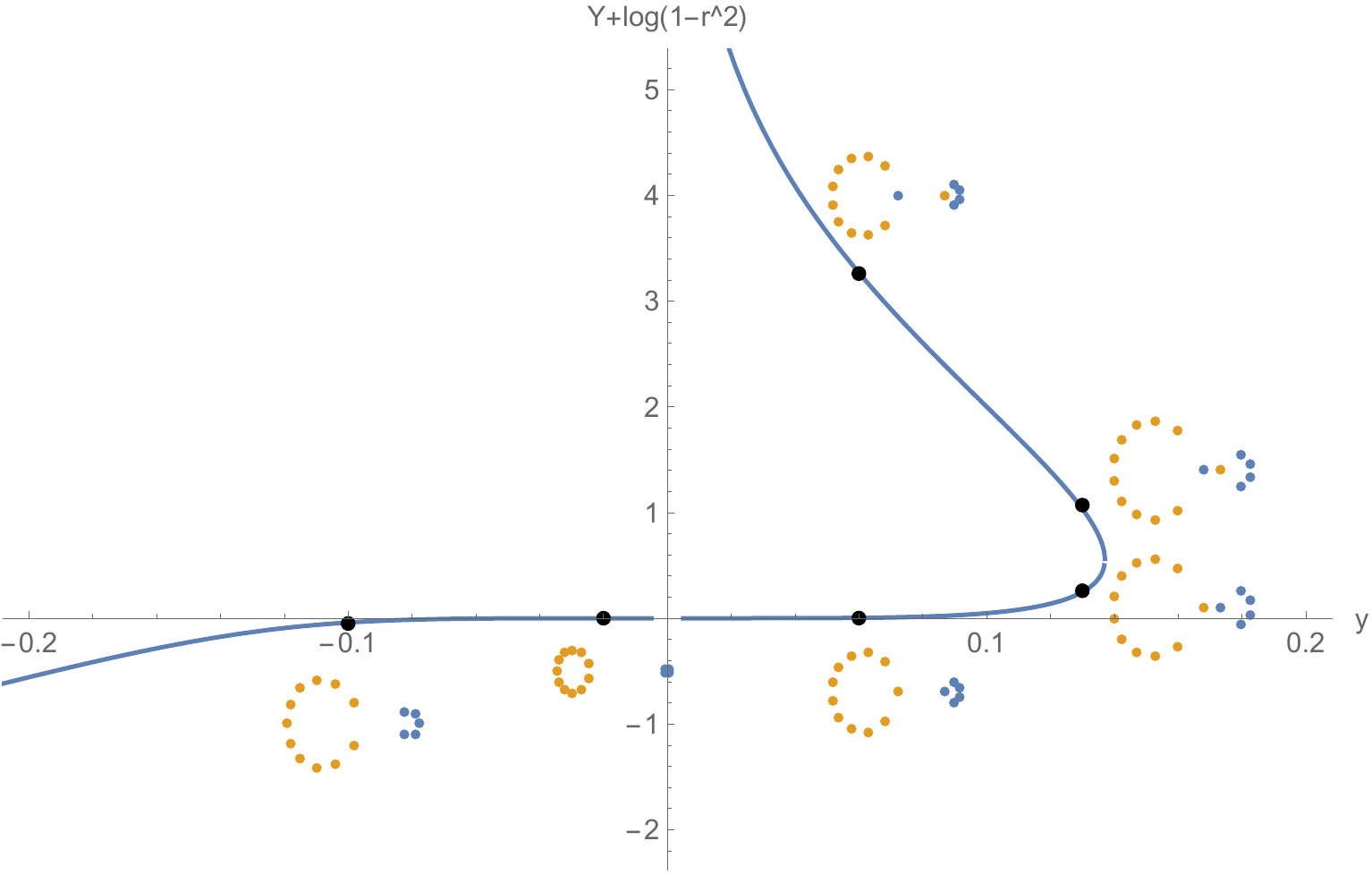}
  \caption{\label{noseplot}A plot of $\text{sign}(y)|y|^{1/n}$ versus $Y+\log(1-r^2)$ for $N=16,n=5$ and $r<1$.
At each black sample point there
is an insert showing the corresponding roots $w_j$ (blue) and other roots (yellow) of the Bethe equation $p(w)=y$.
Note that along the upper branch of the curve, a blue and yellow root have changed places, compared to the lower branch at the same $y$ value.
The upper branch consists of values of $Y+\log(1-r^2)$ which are unobtainable using the $n$ right-most roots, for any $y$, but can be obtained by taking the $n-1$ rightmost roots of the small component of the Cassini oval, and the right-most
root of the left component.}
  \label{dense}
\end{figure}

We now consider the case of $n$ fixed, and
where the Cassini oval has two components. By Lemma \ref{Cassiniprops}(i),
this occurs for small $|y|$. 
Recall that when the Cassini oval has two components, the $n$ roots of $p(w)=y$ of maximal real part are exactly those on the component
surrounding $1$. 

As discussed in Figure \ref{noseplot}, there are two possible situations. In the first, $w_1,\dots,w_n$ are the
roots on the right-most component, that is, those of maximal real part.
We will show that in this case $Y$ tends to a fixed value in the limit $N\to\infty$, namely, 
$Y=-\log(1-r^2)$. In the second case, $y>0$, and $w_1,\dots,w_{n-1}$ are the right-most $n-1$ roots on the 
right component, and $w_n$ is the right-most root on the left component. In this case in the limit of large $N$
$Y$ tends likewise to 
$$Y=-\log(1-r^2)+\log|w_n| - \log|w_n^*|\ge -\log(1-r^2)$$
where $w_n^*$ is the left-most root on the right component. By Lemma \ref{frozen} this second case is irrelevant,
since for these values of $Y$ the system is frozen. This solution of the Bethe equation does not correspond to the maximal eigenvalue.

Let us evaluate (\ref{eq:cdef}) in the first case. Let $C_\rho$ be a circle around the origin of radius $\rho$, where $\rho$ is chosen so that one oval is enclosed and one is in the exterior. By Lemma \ref{lem:Mahler},
\begin{equation}
\label{q1}
  -\log((1-r^2)\e^{Y}) = \frac1n\left(\frac1{2\pi}\int_{0}^{2\pi}\log|p(\rho \e^{i\theta})-y|\,d\theta-(N-n)\log\rho\right).
\end{equation}
Recall that $|p(z)|= |y|$ for $z$ on the Cassini oval. For large $N,n$, the quantity
$|p(z)|$ is constant on the Cassini ovals $\CO_{N-n,n,*}$ in the same family as $|y|$ varies, but increases rapidly as $|y|$ increases. 
In particular $|p(\rho \e^{i\theta})|\gg |y|$ on 
the circle $C_\rho$, so the integrand is well approximated 
by $\log|p(\rho \e^{i\theta})|$. Plugging in $p(w)=w^{N-n}(1-w)^n$ and ignoring the $y$, the right-hand side vanishes to order $o(1)$, giving 
$-\log((1-r^2)\e^{Y})=o(1),$
or, $Y=-\log(1-r^2)+o(1)$. 

To compute the eigenvalue \eqref{eq:eval3} we need to compute 
\begin{equation}\label{denomprod}\prod_{i=1}^n
  \left(1-(1-r^2)w_i\right).\end{equation}
 
The map $w\mapsto 1-(1-r^2)w$ takes points of the oval outside of $C_\rho$ to points \emph{inside}
the disk $C_R$ at $0$ of radius $R=|1-(1-r^2)\rho|$. 

In terms of the polynomial $p(w)$ these points are roots of the monic polynomial
$$(1-r^2)^N\left(p\left(\frac{1-u}{1-r^2}\right)-y\right).$$ 
Its constant coefficient is 
$$c_0=(1-r^2)^N\left(p\left(\frac1{1-r^2}\right)-y\right) = (-1)^nr^{2n}-y(1-r^2)^N$$
which is (up to sign) the product of the $N$ roots.

Thus using Lemma \ref{lem:Mahler}
(with $u=R\e^{i\phi}$)
the log of the product in (\ref{denomprod}) is
\begin{equation}
\label{dp2}
\log\left(\prod_{i=1}^n\left(1-(1-r^2)w_i\right)\right)=
\log|c_0|+n\log R-\frac1{2\pi}\int_{0}^{2\pi} \log\left|(1-r^2)^N\left(p\left(\frac{1-u}{1-r^2}\right)-y\right)\right|\,d\phi.
\end{equation}

Again $y$ is negligible, so this reduces to
\begin{align}
=&\nonumber\log|c_0|+n\log R-\frac1{2\pi}\int_{0}^{2\pi} \log\left|(1-r^2)^N\left(\frac{1-u}{1-r^2}\right)^{N-n}\left(\frac{r^2-u}{1-r^2}\right)^n\right|\,d\phi + o(N).\\
=&\label{mid}\log|c_0|+n\log R-\frac1{2\pi}\int_{0}^{2\pi}(N-n)\log|1-u|+n\log|r^2-u|\,d\phi+ o(N)\\
=&\nonumber\log|c_0|+ o(N), 
\end{align}
where we used $r^2<R<1$ in the last equality.

Plugging in to (\ref{eq:eval3}) yields 
\begin{prop}\label{largeY} Fix $r<1$ and suppose the Cassini oval has two components. Then as $n,N\to\infty$ with ratio tending to $s$, we necessarily have $Y=-\log(1-r^2)$ and
$\lim_{n\to\infty}\frac1N\log\Lambda = Y(1-s)+Xs.$
\end{prop}

\subsection{Cassini oval one component case}

Fix $y$, and define a positive number $\rho$ so that the circle $C_\rho$ of radius $\rho$ centered at the origin surrounds $N-n$ of the roots of $p(w)-y$ and passes through none of the roots. Define $w_0$ to be the intersection point
in the upper half plane 
of $C_\rho$ and the Cassini oval $|p(w)|=|y|$. Then, applying Lemma~\ref{lem:Mahler} to \eqref{eq:cdef}, up to terms tending to zero as $N\to\infty$, (\ref{q1}) gives
\[
  -\log(|1-r^2|\e^{Y})= -\frac{N-n}{n}\log|w_0|+\frac{\arg w_0}{n\pi}\log|y| + \frac1{2\pi n}\int_{\arg w_0}^{2\pi-\arg w_0}\log|p(\rho \e^{i\theta})| d\theta, 
\]
and plugging in $p(w)=w^{N-n}(1-w)^n$ and $|y|=|p(w_0)|$ and simplifying  
\begin{equation}\label{eq:rfromw}
  -\log(|1-r^2|\e^{Y})=\frac{\arg w_0}{\pi}\log|1-w_0|+\frac1{\pi}\Im\Li(w_0).
\end{equation}

The right-hand side of \eqref{eq:rfromw} reoccurs many times later so we define the function
\begin{equation}
  \label{Bdef}
  B(z) = \frac1{\pi}(\arg z\log|1-z|+\Im\text{Li}(z)).
\end{equation}
See the appendix for some of its properties, as well as properties of the dilogarithm $\text{Li}$.

Note that as $w_0$ tends to a point in $(0,1)$, the Cassini oval develops a pinch point, and $B(w_0)$ tends to $0$, which is consistent
with the two-component case.

\subsection{Relating $y$ to $Y$, one component case}

The point $w_0$ of the previous section plays an important role in the rest of the paper; it is a \emph{conformal coordinate} for the model.
Equation \eqref{eq:rfromw} relates $Y$ to a certain function of $w_0$, which itself is defined from $s=n/N$ and $y$. 
We show in this section that 
$w_0$ (and hence $y$) can be determined from $Y$ and $s=n/N,$ and conversely,
$w_0$ determines $Y$ and $s=n/N$. Moreover $y$ and $Y$ are monotonically related.

Consider, for given $n,N$, the triangles with vertices $0,1,w$ with $w$ in the upper half plane, which have angles at $0$ and $1$ with 
ratio $\frac{n}{N-n} = \frac{s}{1-s}$.
These points $w$ form a curve in $\H$.
As $0<s<1$ varies these curves foliate the upper half plane, see Figure \ref{twoanglefoliation}. 
\begin{figure}[htbp]
  \centering
  \includegraphics[width=7cm]{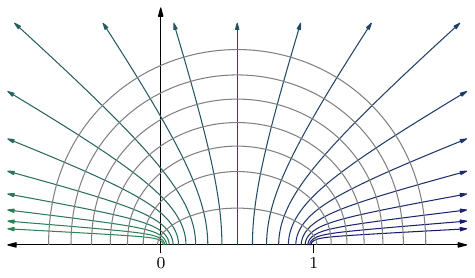}
  \caption{\label{twoanglefoliation}In blue/green, curves of the
    form $\displaystyle{(\arg w)/\arg \tfrac{1}{1-w}} = \tfrac{s}{1-s}$
    for various values of $s\in (0,1)$. The orthogonal family
    of gray curves shown is a collection of level curves of $B$.}
  \label{fig:foliation}
\end{figure}

Given a point $w$ on this curve, $w$ lies on a unique Cassini oval $\{w~:~|w|^{N-n}|1-w|^n=y\}$
with the property that the number of roots of  
$w^{N-n}(1-w)^n=y$ on the right side of the curve is $n$ (see Lemma \ref{Cassiniprops}(ii)).

Thus for any $w_0$ in the upper half plane there is a corresponding
choice of $n/N$: we have 
$$\frac{s}{1-s} = \frac{n}{N-n} =\frac{\arg(w_0)}{\arg(1/(1-w_0))}.$$

Conversely, given $s=n/N$, there is a curve $\frac{\arg(w)}{\arg(1/(1-w))}=\frac{s}{1-s}$, and on this curve
the values of $B(w)$ are monotone running from $0$ (on the real axis in $[0,1]$) to infinity 
(the orthogonality of the $B$ curves and the $s$ curves follows from explicit differentiation). 
Thus $w_0$ is defined
uniquely as a function of $s,Y$. Conversely $w_0$ uniquely determines $s,Y$. It is not hard to establish monotonicity of $Y$ for large finite $N,n$ as well:
\begin{lemma}\label{mono} Suppose the Cassini oval has one component and is not 
at the pinch point, that is, is a simple smooth curve. For sufficiently large $N,n$ and $r<1$, 
$Y$ is a monotone increasing function of $|y|$.
\end{lemma}

\begin{proof} Differentiating $p(w)=y$ we have
\be
  \left(N-n-\frac{nw}{1-w}\right)\frac{dw}{w} = \frac{dy}y\label{eq:2}
\ee 
which implies 
$$\frac{d\log |w|}{d\log y} =\Re\frac1{N-n-\tfrac{nw}{1-w}}.$$
The map $\psi(w)=w^{(N-n)/n}(1-w)$ maps the Bethe roots $w_1,\dots,w_n$ to the
consecutive $n$th roots of $y$. Let $z=\psi(w)$.
Then starting from (\ref{eq:cdef}),
\begin{align*}dY &= -\frac1n\sum_{j=1}^n \frac{dw_j}{w_j}\\
&= -\frac{dy}{yn}\sum_{j=1}^n\frac1{N-n-\tfrac{nw_j}{1-w_j}} \\
&\approx -\frac{dy}{yn}\frac{n}{2\pi i}\int\frac1{N-n-\tfrac{nw}{1-w}}\frac{dz}{z}\\
&= -\frac{dy}{yn}\frac{1}{2\pi i}\int_{\overline{w_0}}^{w_0} \frac{dw}{w}\\
&=-\frac{dy}{yn}\frac{\arg w_0}{\pi}.
\end{align*}
The approximation of the sum by the integral is valid for sufficiently large $n$, and $\arg w_0>0$ as long as we are not at the pinch point.
So with Lemma~\ref{lem:y<0} it follows that $dY/dy$ is always positive.
\end{proof}

We define $\theta$ so that the argument of $w_0$ is $s\theta$, see Figure \ref{wztriangles}.

\begin{figure}[htbp]
\center{\includegraphics[width=2in]{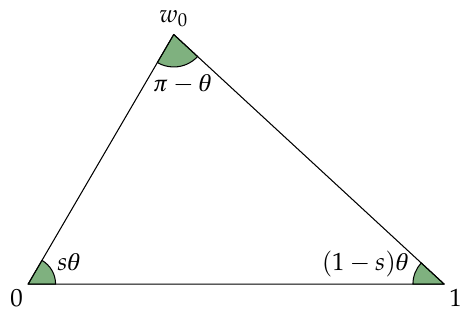}}
\caption{The relationship between $s$, $\theta$, and
  $w_0$. \label{wztriangles}}
\end{figure}

\subsection{The free energy, one component case} 

We can start from (\ref{dp2}).
We can evaluate the integral as follows. 
For points $u$ inside the image of the oval, the integrand is approximated by the quantity
$\log((1-r^2)^N|y|)$, and for points outside
the integrand is approximated by  
$$\log\left|(1-r^2)^Np\left(\frac{1-u}{1-r^2}\right)\right|=\log(|1-u|^{N-n}|r^2-u|^n).$$
The errors in these approximations are both $o(1)$ as $N\to\infty$
except when $u$ is within $o(1)$ of the oval; as such the errors
contribute at most $o(1)$ to the result.  Setting
\be\label{udef}u_0=R\e^{i\phi_0}=1-(1-r^2)\overline w_0
\ee
we have that (\ref{dp2}) is
{\footnotesize$$2n\log r+\log[1-(-1)^nyr^{-2n}(1-r^2)^N]+n\log
  R-\frac{\phi_0}{\pi}(N\log(1-r^2)+\log|y|)-
  \frac{N-n}{2\pi}\int_{\phi_0}^{2\pi-\phi_0}\log|1-u|d\phi -
  \frac{n}{2\pi}\int_{\phi_0}^{2\pi-\phi_0}\log|r^2-u|d\phi
$$}
and splitting up the $2n\log r$ term, 
{\footnotesize
$$=\log[1-(-1)^nyr^{-2n}(1-r^2)^N]+n\log R-\frac{\phi_0}{\pi}(-2n\log r + N\log(1-r^2)+\log|y|)
-\frac{N-n}{2\pi}\int_{\phi_0}^{2\pi-\phi_0}\log|1-u|d\phi - 
\frac{n}{2\pi}\int_{\phi_0}^{2\pi-\phi_0}\log|1-\frac{u}{r^2}|d\phi
$$
}
$$=\log[1-(-1)^nyr^{-2n}(1-r^2)^N]+n\log R-\frac{\phi_0}{\pi}(-2n\log r +N\log(1-r^2)+\log|y|)
-\frac{N-n}{\pi}\Im\Li(u_0) - 
\frac{n}{\pi}\Im\Li\left(\frac{u_0}{r^2}\right).
$$

Plugging in $\log|y|=(N-n)\log|w_0|+n\log|1-w_0|$ yields
\begin{multline*}
=\log[1-(-1)^nyr^{-2n}(1-r^2)^N]+n\log R-\frac{\phi_0}{\pi}((N-n)\log|(1-r^2)w_0|+n\log|\frac{1-r^2}{r^2}(1-w_0)|)\\
  -\frac{N-n}{\pi}\Im\Li(u_0)
-\frac{n}{\pi}\Im\Li\left(\frac{u_0}{r^2}\right)\end{multline*}
$$
=\log[1-(-1)^nyr^{-2n}(1-r^2)^N]-\frac{(N-n)}{\pi}(\phi_0\log|1-u_0|+\Im\Li(u_0))
  -\frac{n}{\pi}\left(\phi_0\log\left|1-\frac{u_0}{r^2}\right| + \Im\Li\left(\frac{u_0}{r^2}\right)\right)+n\log R.
$$

Using \eqref{eq:eval3} we get, in terms of the function $B(z)$ of (\ref{Bdef}),
\begin{equation} \lim_{N\to\infty}\frac1{N}\log\Lambda = sX+(1-s)Y+2s\log r + (1-s)B(u_0)+sB\left(\frac{u_0}{r^2}\right)-s\log|u_0|.
\end{equation}

Using the identity $B(z)-\log|z| = B(1-1/z)$ with $z=u_0/r^2$ gives the microcanonical free energy (\ref{microc}) to be
\begin{equation} F_m(s,Y) = (1-s)Y+(1-s)B(u_0)+sB\left(1-\frac{r^2}{u_0}\right).
\end{equation}

The calculation of the surface tension $\sigma(s,t)$ and free energy $F(X,Y)$, as discussed in (\ref{grandc}),
are performed in Section \ref{sec:legendre} below.

\section{Legendre transform, $r<1$ case}\label{sec:legendre}

The surface tension $\sigma(s,t)$ satisfies 
$$-\sigma(s,t) = F_m(s,Y)-Yt$$ 
where $t=\frac{dF_m}{dY}$.

Fixing $r, s$, we can parameterize $F_m(s,Y)$ explicitly in terms of $\theta$, where $s\theta=\pi-\arg(u_0-1),$ 
see Figure \ref{trianglefig}.
Then $Y$ is a function of $\theta$ given by \eqref{eq:rfromw}.

Using $Y=-\log(1-r^2)-B(w_0),$ we compute using (\ref{dB})
\begin{align*}\frac{dY}{d\theta} &= -\frac{dB(w_0)}{d\theta}\\
&=-\frac{s\theta}{\pi} \frac{d}{d\theta}\log|1-w_0| - (1-s)\frac{\theta}{\pi} \frac{d}{d\theta}\log|w_0|.
\end{align*}

Letting $\phi=\arg u_0$ we have
$$\frac{dB(u_0)}{d\theta} = \frac{\phi}{\pi}\frac{d}{d\theta}\log|1-u_0|+ \frac{s\theta}{\pi}\frac{d}{d\theta}\log|u_0|$$
and
\begin{align*}
\frac{dB(1-\frac{r^2}{u_0})}{d\theta} &= \frac{\phi}{\pi}\frac{d}{d\theta}\log|(u_0-r^2)/u_0|+
\frac{((1-s)\theta-\phi)}{\pi}\frac{d}{d\theta}\log|r^2/u_0|\\
&= \frac{\phi}{\pi}\frac{d}{d\theta}\log|u_0-r^2| -(1-s)\frac{\theta}{\pi}\frac{d}{d\theta}\log|u_0|.
\end{align*}

Then 
\begin{align*}\frac{dF_m(s,Y)}{d\theta} &= (1-s)\frac{dY}{d\theta}  +(1-s)\frac{dB(u_0)}{d\theta}+s\frac{dB(1-\frac{r^2}{u_0})}{d\theta}\\
&= (1-s)\frac{dY}{d\theta}  +(1-s)\frac{\phi}{\pi}\frac{d}{d\theta}\log|1-u_0| +\frac{s\phi}{\pi}\frac{d}{d\theta}\log|u_0-r^2|\\
&= (1-s)\frac{dY}{d\theta}  +(1-s)\frac{\phi}{\pi}\frac{d}{d\theta}\log\frac{|1-u_0|}{1-r^2} +\frac{s\phi}{\pi}\frac{d}{d\theta}\log\frac{|u_0-r^2|}{1-r^2}\\
&= (1-s)\frac{dY}{d\theta}  +\frac{\phi}{\pi}\left((1-s)\frac{d}{d\theta}\log|w_0| +s\frac{d}{d\theta}\log|1-w_0|\right)\\
&= \left(1-s-\frac{\phi}{\theta}\right)\frac{dY}{d\theta}.\\
\end{align*}

Thus 
$$t=\frac{dF_m}{dY} = \frac{dF_m}{d\theta}/\frac{dY}{d\theta} = 1-s-\frac{\phi}{\theta},$$
and so
$\phi = (1-s-t)\theta$.

Figure \ref{trianglefig} illustrates the relationship between $s$,
$t$, $u$, $r$, and $\theta$. 
\begin{figure}[htbp]
\center{\includegraphics[width=4in]{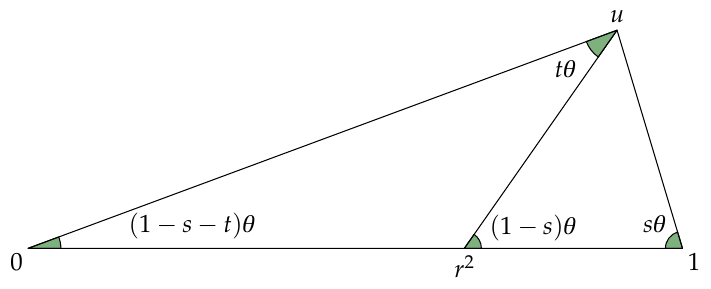}}
\caption{\label{trianglefig}Given $r$ and the ratios $[s:t:1-s-t]\in \N^*$, there is a unique $u=u_0$ in the upper half plane 
satisfying the property
that the angles indicated are in the ratio $[s:t:1-s-t]$. }
\end{figure}

Note that not all ratios $[s:t:1-s-t]$ are feasible: as $u_0$ ranges over the upper half plane, a short computation
shows that the values $(s,t)$ range over the subset $\N^*$ of $\N$ bounded by the axes and the hyperbola 
\begin{equation} 
\label{Nbound}
\left(\frac{1-r^2}{r^2}\right)st+s+t-1=0,
\end{equation}
see Figure \ref{hyperb}: the three intervals $u\in(1,\infty)$, $u\in(-\infty,0)$, $u\in(0,r^2)$ map respectively to the 
vertices $(s,t)=(1,0), (0,0)$ and $(0,1)$. For $r^2<u<1$, taking the limits as $\Im u\to0$ of the angle ratios, we have $s=\frac{u-r^2}{1-r^2}$ and $t=\frac{r^2(1-u)}{(1-r^2)u},$ which parameterizes the curved edge of $\N^*$.
\begin{figure}[htbp]
\center{\includegraphics[width=2in]{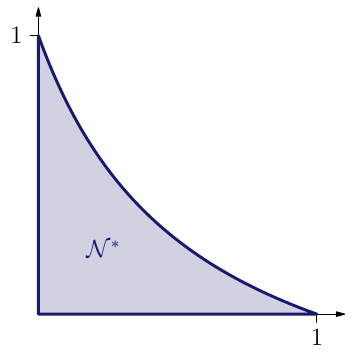}}
\caption{\label{hyperb}The region $\N^*$ is bounded by the axes and the hyperbola $(\frac{1-r^2}{r^2})st+s+t-1=0$ (shown here for $r=0.6$).
}
\end{figure}
When $r\to 1$ this hyperbola degenerates to the line $s+t=1$. 

For $(s,t)\in \N\setminus \N^*$, there is not a unique Gibbs measure with slope $(s,t)$; see the discussion after 
Proposition \ref{stension} below.

Recalling that $\overline u_0=1-(1-r^2)w_0$, we have determined the
relation between $w_0$ and $(s,t)$.  We already have the relation
between $Y$ and $w_0$, which is
$$Y=-\log|1-r^2|-B(w_0).$$
By symmetry 
$$X=-\log|1-r^2|-B(\bar z_0)$$
where $\bar z_0$ is defined as for $w_0$ but with the roles of $s$ and $t$ reversed.

Plugging in gives the surface tension and free energy: 
\begin{prop}
\label{stension}
The surface tension is given by
$$\sigma(s,t)  = (1-s-t)(\log(1-r^2)+B(w_0))-(1-s)B(u_0)-sB\left(1-\frac{r^2}{u_0}\right)$$
where $u_0$ is determined by $(s,t)$ implicitly as above, and $w_0=(1-\overline u_0)/(1-r^2).$
The free energy is given by 
$$F(X,Y) = -(1-s)Y-(1-s)B(u_0)-sB(1-\frac{r^2}{u_0})+sX.$$
\end{prop}

A plot of $\sigma$ is shown in Figure \ref{surftension}.
\begin{figure}[htbp]
\center{\includegraphics[width=4in]{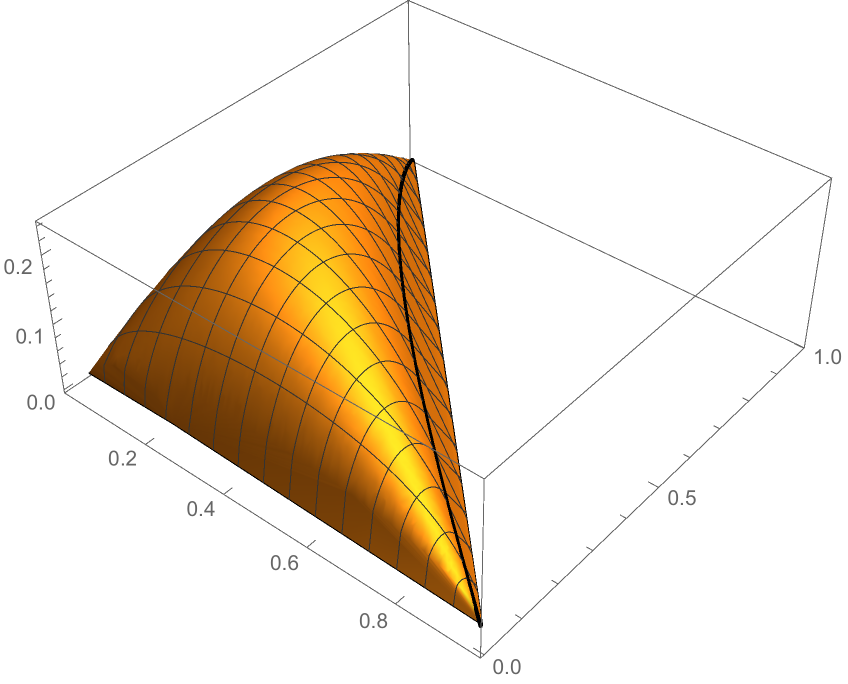}}
\caption{\label{surftension}Minus surface tension as a function of $(s,t)\in \N$ with $r=0.8$. The black line is graph of $-\sigma$ on the hyperbola bounding $\N^*$; $\sigma$ is linear
in $\N\setminus \N^*$.}
\end{figure}
A plot of the free energy $F(X,Y)$ is shown in Figure \ref{freeenergy}.
\begin{figure}[htbp]
\center{\includegraphics[width=4in]{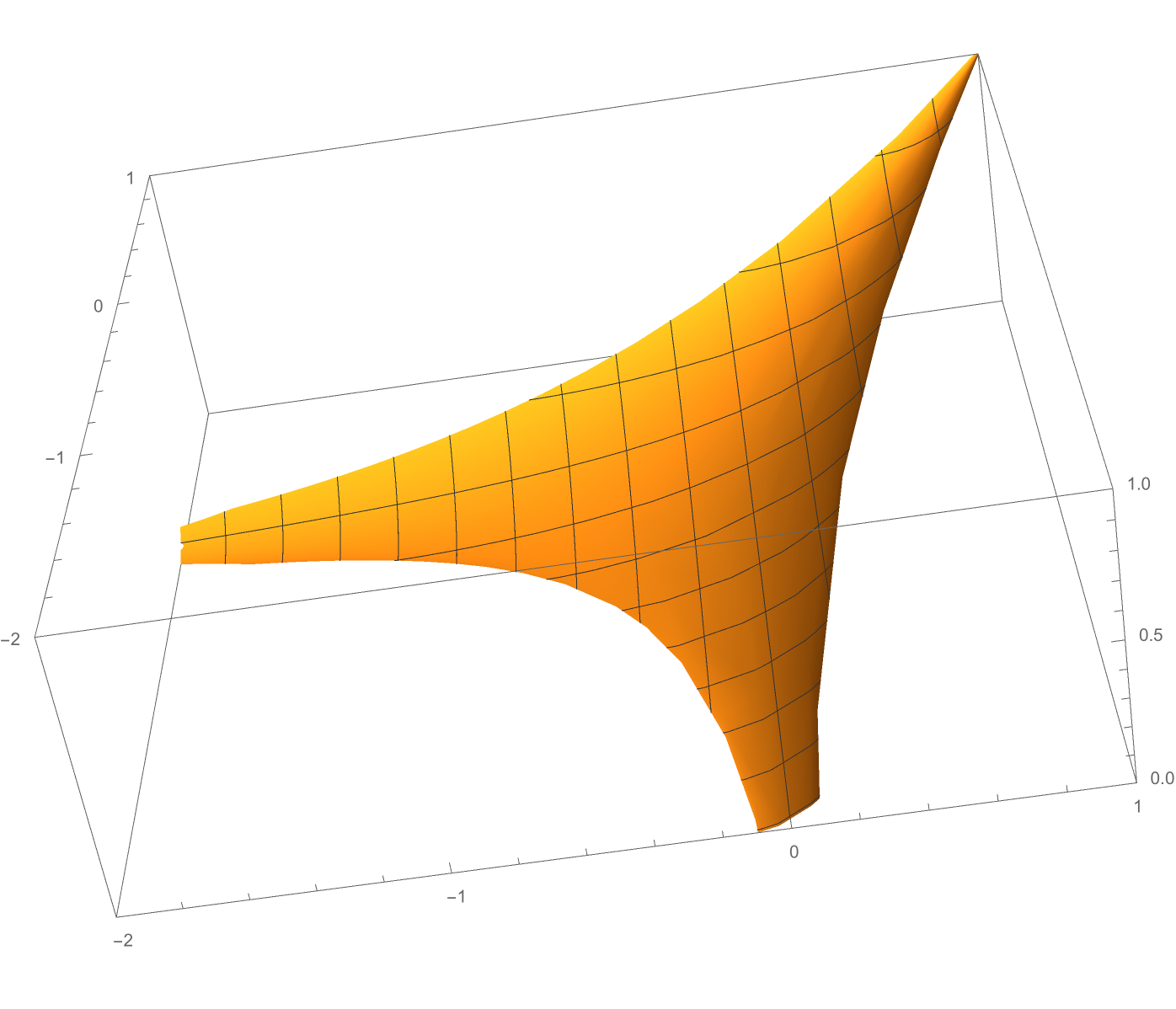}}
\caption{\label{freeenergy}The free energy $F(X,Y)$ for $r=0.8$. 
It is linear in each complementary component of the curved region
shown (except for a slope change along $y=x$).}
\end{figure}

Since $\sigma$ is linear on $\N\setminus\N$, there is not a unique
surface-tension-minimizing Gibbs measure of slope $(s,t)$ in this region:
any convex combination of Gibbs measures for $(s,t)$ on the bounding hyperbola, and of the appropriate slope, 
will be a minimizing Gibbs measure. In the analogous setting in the six-vertex model,
Aggarwal \cite{Aggarwal} showed that there is no ergodic Gibbs measure of such a slope.

\section{Euler-Lagrange equation, $r<1$ case}\label{sec:EL}

In this section we define $w=w_0, z=z_0$ and $u=u_0$, and recall the definition of $B(\cdot)$ in (\ref{Bdef}). 

Assuming the surface tension minimizer $h=h(x,y)$ is $C^2$, the Euler-Lagrange equation for $h$ 
is a PDE for $s=s(x,y), t=t(x,y)$: 
\be\label{ELeqn}\text{div}_{x,y}(\nabla_{s,t}\sigma(s,t))= 0.\ee 
Generally, however, 
 we cannot guarantee that $h$ is $C^2$: the ellipticity of the equation
(the ratio of eigenvalues of the Hessian of $\sigma$) degenerates at the vertices of  $\N$, and this causes the formation
of nonanalyticities in the limit shape called \emph{arctic curves}. These arise even in the case of the dimer model ($r=1$). However a theorem of Morrey \cite{Morrey} says that an analytic surface tension implies analytic solutions. 
Thus in the regions where $(s,t)$ lie in the interior of $\N^*$, the 
limit shape function $h$ is real analytic. One can work around the problem of nonanalyticity by first,
finding an analytic solution to (\ref{ELeqn}) and then extending it to its maximal domain of analyticity; in the cases
analyzed in \cite{kenyon2007limit} this maximal domain is bounded by curves on which $(s,t)$ limits to a corner of $\N$ (the arctic curves). One can then extend the solution linearly with that constant slope across these curves. An analogous
procedure applies in the present case when $r>1$, or when $r<1$ but one avoids $\N\setminus\N^*$. See 
however the discussion after Theorem \ref{ckpthm} and in Section \ref{bppsection} regarding Figures
\ref{bppr} and \ref{bppexact}. Making a more 
general statement requires further arguments which we will not attempt here. Beyond these examples
we will be satisfied in the present paper with
limit shapes having slope in the interior of $\N^*$, where we can use (\ref{ELeqn}).

Since $\sigma_s = X, \sigma_t=Y$, the Euler-Lagrange equation can be written 
\begin{equation}
\label{EL1}X_x+Y_y=0.\end{equation}
Note that we also have the equation of mixed partial derivatives of $h$:
\begin{equation}
\label{EL2}s_y=t_x.\end{equation}

We already have $\frac{\partial\sigma}{\partial t}=Y=-\log(1-r^2)-B(w)$. 
By symmetry $\sigma(s,t)=\sigma(t,s)$ which gives
$$\frac{\partial\sigma}{\partial s}=X = -\log(1-r^2)-B(\bar z)=-\log(1-r^2)+B(z).$$ 

\begin{prop}[Euler-Lagrange equation]
 \label{prop:EL}
For $r<1$ the EL equation (\ref{EL1}), combined with the mixed partials equation (\ref{EL2}), reduces to 
\begin{equation}
\label{EL3}B(z)_zz_x-B(w)_ww_y=0.\end{equation}
\end{prop}

Note that since $B$ is not complex analytic this equation is NOT equivalent to (\ref{EL1});  (\ref{EL1}) alone implies only that
$$B(z)_zz_x + B(z)_{\bar z}\bar z_x - B(w)_ww_y - B(w)_{\bar w}\bar w_y=0.$$

To prove the proposition we start with a lemma.
\begin{lemma}\label{spcurvesmall}
The quantities $w,z$ satisfy the relation
$p(w,z):=1-w-z+(1-r^2)wz=0.$
\end{lemma} 

\begin{proof}Recall that $u=1-(1-r^2)\overline w$; let $u^*:=1-(1-r^2)\overline w^*=1-(1-r^2)z$.  Referring to Figure \ref{similartris}, we see that $\arg u=\arg u^*$ and $|u|\cdot |u^*|=r^2$. Thus $u^*\overline u=r^2$, or $(1-(1-r^2)z)(1-(1-r^2)w)=r^2,$ which is the desired result.
\end{proof}

\begin{figure}[htbp]
\center{\includegraphics[width=4in]{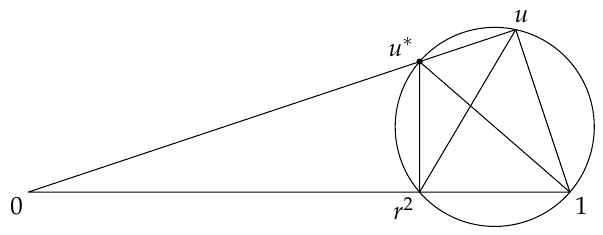}}
\caption{\label{similartris}Exchanging the roles of $s$ and $t$ in Figure \protect{\ref{trianglefig}}, 
the quantity $u^*$ can be obtained from $u$
by inversion in the ball of radius $r$ around the origin. Equivalently, $u,u^*$ have the same argument and the angle at $u$ of the triangle $\angle{1,u,r^2}$ is the same as for that of $u^*$, that is, the points $1,r^2,u,u^*$ are concentric.}
\end{figure}

\begin{proof}[Proof of Proposition~\ref{prop:EL}]
We compute with $w=w_1+iw_2$ (and using (\ref{By}))
\begin{align*}Y_y &= -\frac{dB(w)}{dw_1} \frac{d w_1}{dy} - \frac{dB(w)}{dw_2} \frac{dw_2}{dy} \\
&=\left(\arg w\Re\frac1{1-w}+\arg(1-w)\Re\frac1w\right)\frac{d w_1}{dy} + \left(-\arg w\Im\frac1{1-w}-\arg(1-w)\Im\frac1w\right)\frac{dw_2}{dy}\\
&=\arg w\Re\left(\frac{w_y}{1-w}\right)+\arg(1-w)\Re\left(\frac{w_y}{w}\right).
\end{align*}
where $w_y=\frac{\partial}{\partial y}(w_1+iw_2)$, and similarly
$$X_x = -\arg z\Re\left(\frac{z_x}{1-z}\right)-\arg(1-z)\Re\left(\frac{z_x}{z}\right).$$
Thus $X_x+Y_y=0$ becomes
\begin{equation}\label{XYexpand}-\arg z(\log|1-z|)_x+\arg(1-z)(\log|z|)_x+\arg w(\log|1-w|)_y-\arg(1-w)(\log|w|)_y=0.\end{equation}

We can write (\ref{EL2}) as
\begin{equation}\label{sytx}\frac{\partial}{\partial x}\left(\frac{\arg z}{\arg(\frac{z}{1-z})}\right) =\frac{\partial}{\partial y}\left(\frac{\arg w}{\arg(\frac{w}{1-w})}\right).\end{equation}
Let 
$$d=\arg\frac{z}{1-z}=-\arg\frac{w}{1-w}.$$ 
Multiplying (\ref{sytx}) by $d^2$ gives
$$\arg(z)_x d - \arg(z)d_x =-\arg(w)_y d + \arg(w)d_y$$
or
\begin{equation}\label{stexpand}-\arg(z)_x \arg(1-z) +\arg(1-z)_x \arg(z) +\arg(w)_y \arg(1-w) - \arg(1-w)_y\arg w=0.\end{equation}

Adding $-i$ times (\ref{stexpand}) to (\ref{XYexpand}) gives
$$\arg w(\log(1-w))_y-\arg(1-w)(\log w)_y-\arg z(\log(1-z))_x+\arg(1-z)(\log z)_x=0$$
or, using (\ref{Bz}),
\begin{equation}
\label{Burg}\frac{\partial B(w)}{\partial w} w_y - \frac{\partial B(z)}{\partial z} z_x = 0,\end{equation}
the desired result.
\end{proof}

It seems fortuitous that we can reduce the two real equations (\ref{EL1}),(\ref{EL2}) 
to a single complex equation; however
this is a general phenomenon for variational problems of this type, due to Amp\`ere, 
see \cite{Kenyonanalytic}.

We can simplify equation (\ref{Burg}) as follows.
Using the fact that $\frac{dw}{w(1-w)}=-\frac{dz}{z(1-z)}$ (which follows from Lemma \ref{spcurvesmall}), it becomes
$$\frac{\partial B(w)}{\partial w} w(1-w)w_y +\frac{\partial B(z)}{\partial z} z(1-z)w_x =0$$
and with (\ref{Bz})
\begin{equation}
\label{BB0}(w\arg w+(1-w)\arg(1-w))w_y + (z \arg z+(1-z)\arg(1-z))w_x=0.\end{equation}

\begin{lemma}
Let $r<1$. Let $\A$ be the function 
$$\A(z) = \frac1{\pi}(-z\arg z-(1-z)\arg(1-z)) = 2\frac{\partial B(z)}{\partial z} z(1-z).$$ 
The $(x,y)$ coordinates of the limit shape satisfy the linear first-order PDE
\begin{equation}\label{argent}\A(w)x_{\overline w} - \A(z)y_{\overline w}=0.\end{equation}
\end{lemma}
\begin{proof}
This follows immediately from (\ref{BB0}), the (locally defined) inverse mapping $x=x(w),y=y(w)$, and the identity $w_xx_{\overline w}+w_yy_{\overline w}=0$. (While the map from $(x,y)$ to $w$ may not be globally invertible,
it is an open map hence invertible locally except at isolated branch points.)
\end{proof}

\section{Limit shapes}\label{sec:limitshapes}
In this section we derive a parametrisation for $(x,y,H)$, the coordinates and limiting shape of the height function, in terms of an analytic function describing the boundary domain. The main results are Theorem~\ref{xyH} in Section~\ref{se:general} and its corollary in Section~\ref{se:arctic}.

\subsection{Integration}

The equation (\ref{argent}) can be solved analytically as follows.
Divide (\ref{argent}) by $-\theta= \arg(z/(1-z))=-\arg (w/(1-w))$ to get  
$$\frac{w\arg w+(1-w)\arg(1-w)}{\arg w-\arg(1-w)}x_{\overline w} + \frac{z \arg z+(1-z)\arg(1-z)}{\arg z-\arg (1-z)}y_{\overline w}=0,$$
or (recalling $s=\frac{\arg w}{\arg(w/(1-w))}$ and $t=\frac{\arg z}{\arg(z/(1-z))}$)
\begin{equation}
\label{deriv}(w-1)x_{\overline w}+ sx_{\overline w} +(z-1)y_{\overline w} + ty_{\overline w}=0.\end{equation}

Since $s=H_x,t=H_y$ for the (real valued) height function $H$ we have
$sx_{\overline w} +ty_{\overline w}=H_{\overline w}.$

The equation (\ref{deriv}) is then
$$(w-1)x_{\overline w}+ (z-1)y_{\overline w} + H_{\overline w}=0,$$ which we can integrate to get
\begin{theorem} The following equation holds:
\begin{equation}\label{Heq} (w-1)x+(z-1)y+H(x,y) + f(w) = 0\end{equation}
where $f$ is an analytic function depending only on boundary conditions.
\end{theorem}

Remarkably, the same result holds in the $r>1$ case as well, see Theorem \ref{univlarger}.

Taking the imaginary part of (\ref{Heq}) gives
\begin{equation}\label{wzf}\Im[(w-1)x+(z-1)y] =- \Im f(w)\end{equation} 
where the right-hand side is an arbitrary harmonic function.

Write $w=\frac{1-\overline u}{1-r^2}$ and $z=\frac{1-u^*}{1-r^2}$ with $u^*=r^2/\overline{u}$. 
Then with $u=\RR \e^{i\phi}$ equation (\ref{wzf}) gives (absorbing the $(1-r^2)$ factor into $f$)
\begin{equation}\label{yfromx}x-\frac{r^2}{\RR^2}y = \frac{g(u)-\overline{g(u)}}{u-\overline u}\end{equation}
where $g(u)=-(1-r^2)f(w)$. For any given $g$ this can be solved for $y$, and, plugging back into (\ref{argent}),
gives an equation of the form $$A_1x_u+ A_2x =A_3$$ where
$A_1,A_2,A_3$ are functions of $u$. This can then be integrated by standard techniques.

\subsection{Example zero}

Here is an example. Suppose $g\equiv 0$. 
Starting from (\ref{argent}), using $x_{\overline w} = -(1-r^2)x_u$ and $y_{\overline w} = -(1-r^2)y_u$
gives (using (\ref{yfromx}))
$$\A(w)x_u - \A(z)\left(\frac{\RR^2}{r^2}x\right)_u = 0,$$
or, multiplying by $r^2$, using $\RR^2=u\overline u$ and rearranging gives
\begin{equation}
\label{xuexzero}
\frac{x_u}{x} = \frac{\A(z)\overline u}{r^2\A(w)-u\overline u\A(z)}. 
\end{equation}

\begin{lemma}
 \label{Ader}
The numerator of (\ref{xuexzero}) is minus the $u$ derivative of its denominator:
$$ (r^2\A(w)-u\overline u\A(z))_u = -\A(z)\overline u.$$
\end{lemma}

\begin{proof}
Using equation (\ref{LapB}), we have 
\begin{align*}
&\frac{d}{d\overline w}[r^2B_ww(1-w)-u\overline uz(1-z)B_z] = 
r^2w(1-w)B_{w\overline w}-u\overline uz(1-z)B_{z\overline z}\frac{d\overline z}{d\overline w}+(1-r^2)\overline uz(1-z)B_z\\
&=r^2w(1-w)\frac{\Im(w)}{2\pi|w|^2|1-w|^2}-u\overline uz(1-z)\frac{\Im(z)}{2\pi|z|^2|1-z|^2}(-\frac{\overline z(1-\overline z)}{\overline w(1-\overline w)})
+(1-r^2)\overline uz(1-z)B_z\\
& =(1-r^2)\overline uz(1-z)B_z
\end{align*}
since 
$r^2\Im(w)+u\overline u\Im(z)=0.$ Since $\frac{d}{du} = \frac{-1}{1-r^2}\frac{d}{d\overline w}$, this proves the claim.
\end{proof}

Thus
$$x = \frac{C(\bar u)}{r^2\A(w)-u\overline u\A(z)}$$ 
where $C$ is analytic.
Here $C$ is determined by the fact that $x$ must be real; however the denominator is already real:

\begin{lemma}
\label{rARAreal}
The expression $r^2\A(w)-u\overline u\A(z)$ is real.
\end{lemma}

\begin{proof} We have 
\begin{align*}\pi\Im(r^2\A(w)-u\overline u\A(z))&=\Im[r^2(-w\arg w-(1-w)\arg(1-w))-\RR^2(-z\arg z-(1-z)\arg(1-z))]\\
&=-r^2(\Im(w)\arg(w)+\Im(1-w)\arg(1-w))+\RR^2(\Im(z)\arg(z)+\Im(1-z)\arg(1-z))\\
&=-r^2\Im(w)(\arg w-\arg(1-w)) +\RR^2\Im(z)(\arg z-\arg(1-z))\\
&=-\theta(r^2\Im w+\RR^2\Im(z))\\
&=0. \qedhere
\end{align*}
\end{proof}

So in this case $C$ is a real constant. Setting $y=\frac{\RR^2}{r^2}x$, we arrive at:
$$(x,y) = \left(\frac{C}{r^2\A(w)-u\overline u\A(z)},\frac{Cu\bar u}{r^2(r^2\A(w)-u\bar u\A(z))}\right).
$$

\subsection{General solution}
\label{se:general} 

For general $g$, using again (\ref{yfromx}), (\ref{argent}) is
$$\A(w)x_u - \A(z)\left(\frac{\RR^2}{r^2}(x-k)\right)_u = 0,$$
where $k(u)= \frac{g(u)-\overline{g(u)}}{u-\overline u}$.
Multiplying by $r^2$ and expanding,
$$(r^2\A(w)-\RR^2\A(z))x_u - \overline u\A(z)x = -\A(z)(\RR^2k_u+k(\RR^2)_u),$$
or, by Lemma~\ref{Ader} and using $\RR^2=u\overline u$,
$$[(r^2\A(w)-\RR^2\A(z))x]_u = -\A(z)(\RR^2k_u+k\overline u)=-\A(z)(u\overline u k)_u.$$
This can be integrated to yield
\begin{align*}x&=\frac{-1}{(r^2\A(w)-\RR^2\A(z))}\left[\int \A(z)(u\overline u k)_udu + C(\overline u)\right].
\end{align*}
Now upon integration by parts,
$$x = \frac{-1}{(r^2\A(w)-\RR^2\A(z))}\left[\A(z)u\overline u k - \overline u\int uk\frac{d\A(z)}{du}du + C(\overline u)\right].$$
Since
$$\frac{d\A(z)}{du} = \frac{-1}{1-r^2}\frac{d\A(z)}{d\overline w} = \frac{\Im(z)}{\pi(1-r^2)\overline w(1-\overline w)} = -\frac{r^2\Im(u)}{\pi u\overline u(1-u)(u-r^2)},$$
(the second equality here follows from a short calculation)
we have
\begin{align}x &= \frac{-1}{(r^2\A(w)-\RR^2\A(z))}\left[\A(z)\RR^2 k + \frac1{\pi}\int \frac{r^2\Im g(u)}{(1-u)(u-r^2)}du + C(\overline u)\right]\nonumber\\
\label{xgendef}&= \frac{-1}{(r^2\A(w)-\RR^2\A(z))}\left[\A(z)\RR^2 k + \frac1{2\pi i}\int \frac{r^2g(u)}{(1-u)(u-r^2)}du-\frac{\overline{g(u)}}{2\pi i}\int \frac{r^2}{(1-u)(u-r^2)}du + C(\overline u)\right].\end{align}

Here $C(\overline u)$ is determined by the property that $x$ is real:

\begin{lemma}
Up to an additive constant, the integration constant $C(\overline u)$ is given by the following analytic function of $\overline u$:

$$ C(\overline u) = -\frac1{2\pi i}\int\frac{r^2\overline{g(u)}}{(1-\overline u)(\overline u-r^2)}d\overline u+\frac{\overline{g(u)}}{2\pi i}\int\frac{r^2}{(1-\overline u)(\overline u-r^2)}d\overline u
-\frac{r^2 \overline{g(u)}}{1-r^2}.
$$
\end{lemma}

\begin{proof}
To see this, take the imaginary part of the quantity in square brackets in (\ref{xgendef}); the term outside the brackets is real by Lemma \ref{rARAreal}.
The first integral in the second line of (\ref{xgendef}) and the first integral in $C(\overline u)$ add together to give a real quantity.
The second integrals add to give 
$$-\frac{r^2\overline{g(u)}}{\pi(1-r^2)}\arg\left(\frac{u-r^2}{u-1}\right) = -\frac{r^2\overline{g(u)}}{1-r^2}\frac{(-\pi+\theta)}{\pi}.$$
Finally note that 
$$\Im \A(z) = -\tfrac1{\pi}\Im z(\arg z-\arg(1-z))=\tfrac{\theta}{\pi}\Im z = -\frac{r^2\theta\Im(u)}{\pi(1-r^2)\RR^2}$$
so that 
$$\Im(\A(z)\RR^2 k(u)) = -\frac{r^2\theta\Im g(u)}{\pi(1-r^2)}.$$
Using the fact that $\Im g(u) = -\Im \overline{g(u)}$, this completes the proof of the claim.
\end{proof}

Plugging this value of $C(\overline u)$ into (\ref{xgendef}), we can then write 
\begin{equation}\label{x}x=\frac{-1}{(r^2\A(w)-\RR^2\A(z))}\left[\Im\left(\frac1{\pi}\int\frac{r^2 g(u)}{(1-u)(u-r^2)}du\right)+\Re\left(\A(z)\RR^2k(u)-\frac{g(u)r^2\theta}{\pi(1-r^2)}\right)\right].\end{equation}

We collect the parametrisation of the coordinates and limit shape of the height function, $(x,y,H)$, in terms of an analytic function $g(u)$, in a final theorem. For any analytic function $g(u)$, define $\F(u)$ to be the analytic function \begin{equation}
\label{Fdef} \F(u) = \int\frac{r^2 g(u)}{(1-u)(u-r^2)}du. \end{equation} Then $k(u)=\Im g(u)/\Im u$ is written in terms of $\F(u)$ as $$k(u) = \frac{\Im(\F'(u)(1-u)(u-r^2))}{r^2\Im (u)}.$$
Let furthermore, as above, $u=\RR \e^{i\phi}$ and
$$
w=\frac{1-\overline u}{1-r^2},\quad z=\frac{1-r^2/\overline u}{1-r^2},\quad \theta = -\arg \left(\frac{z}{1-z}\right),\quad 
\A(z)=\tfrac1{\pi}(-z \arg(z) - (1-z) \arg(1-z)).
$$

\begin{theorem}\label{xyH}
With the definitions above we have
\begin{equation} \label{xF}
  x=\frac{-1}{(r^2\A(w)-\RR^2\A(z))}\left[\Im
    \tfrac1{\pi}\F(u)+\Re\left(-\frac{\F'(u)(1-u)(u-r^2)\theta}{\pi(1-r^2)}+\A(z)\RR^2k(u)\right)\right],
\end{equation}
and $y$ is determined from (\ref{yfromx}),
\begin{equation} \label{arcticy}
  y= \frac{u \overline u}{r^2} \left (x - k(u) \right).
\end{equation}
The (real) height function $H$ is determined from (\ref{Heq}),
\begin{equation}
  H=\frac{1}{1-r^2}\left( -\overline{g(u)} + (\overline u -r^2)x + r^2(\overline u^{-1}-1)y\right). 
\end{equation}
\end{theorem}

\subsection{Arctic boundary}\label{se:arctic}

Note that in order for the solution to have an `arctic' boundary, equation (\ref{yfromx}) 
needs to have a limit when $u$ becomes real; this means $\Im(g)$ is zero for $u\in\R$, that is, $g$ is \emph{real} analytic. Let us consider the behavior of $(x,y)$ along the arctic boundary, when $u=p+q i$ for $p\in\R$ and $q=\eps$ is small. 

The function $\A(z)$ is piecewise analytic with four different pieces, giving rise to a piecewise analytic boundary. The behaviour along the boundary of the ingredients of (\ref{xF}) are given in Appendix~\ref{ap:arctic}. Let $g(u)$ be an analytic function parametrising a domain via (\ref{xF}) , and define $\F(u)$ as in (\ref{Fdef}). Let $$F=\Re \F,$$ then the analytic boundary pieces of the limit shape are determined, up to local constants $c_1,c_2,c_3$, by the following equations.

\begin{equation}\label{xinF}
x=\begin{cases}
\displaystyle F_p\frac{(r^2-2p)(1-p)^2}{r^2(p^2-2p+r^2)}+F_{pp}\frac{p(r^2-p)(1-p)^2}{r^2(p^2-2p+r^2)} 
- \frac{c_1(1-r^2)}{p^2-2p+r^2}&p<0\\[.15in]
\displaystyle F_p\frac{1-2p}{r^2}+F_{pp}\frac{p(1-p)}{r^2}- \frac{c_2r^2(1-r^2)}{(p-r^2)^2}&0<p<r^2\\[.15in]
\displaystyle F_p\frac{1+r^2-3p}{r^2}+F_{pp}\frac{-r^2+2p+2r^2p-3p^2}{r^2}+F_{ppp}\frac{(1-p)(p-r^2)p}{2r^2}&r^2<p<1\\[.15in]
\displaystyle F_p\frac{r^2-2p}{r^2}+F_{pp}\frac{p(r^2-p)}{r^2}- \frac{c_3(1-r^2)}{(p-1)^2}&1<p,\end{cases}
\end{equation}
where we have used the fact that $\F$ satisfies the Cauchy-Riemann equations. 
The corresponding paramerisation of the $y$ coordinate on the boundary is determined by
\begin{equation}\label{gprime}y=\frac{p^2}{r^2} \left( x - g'(p)\right).
\end{equation}
The constants 
$c_1,c_2,c_3$ come from the imaginary part $\Im\F$, and depend on the height function along the relevant part of the boundary:
when $p<0$, the height function $H=c_1$ is a constant;
when $0<p<r^2$, the height $H$ is of the form $H = y+c_2$ and when $1<p$, $H$ is of the form $H=x+c_3$.

All but the third of the four equations (\ref{xinF}) also follow directly from (\ref{Heq}) and (\ref{yfromx}), which we can 
consider to be two linear equations for $x,y$, once we plug in the appropriate linear function $H$ and take the limit
as $u$ tends to the boundary.

\subsection{Quadratic example}

Let us take $g(u)=-(1-u)(u-r^2)/r^2$ so that $\F(u)=-u$ by (\ref{Fdef}), and therefore on the boundary $F=-p$ 
with $u=p+iq$. The resulting limit shape and piecewise analytic boundary is depicted in Figure \ref{semibpp}.

\begin{figure}[htbp]
\center{\includegraphics[width=4in]{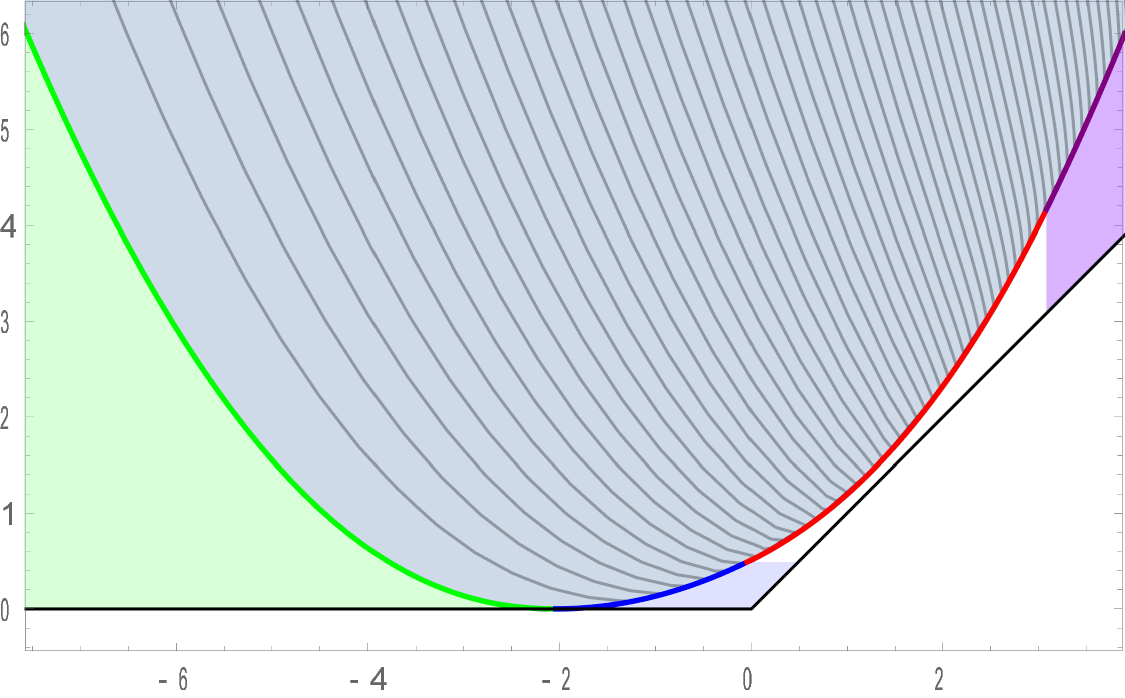}}
\caption{\label{semibpp}The limit shape with $g(u)=-(1-u)(u-r^2)/r^2$ inside the region bounded by the $x$ axis and the line $x=y$, 
shown with height contours (here $r=0.7$).
The four analytic pieces of the frozen boundary are shown in the four colors green, blue, red, purple. }
\end{figure}

The grey region in the figure is the repulsive region. One can extend the height function on the repulsive region to the green, blue, and purple regions by the linear functions $0, y$, and $x$ respectively, so that it 
has boundary values $0$ on the $x$-axis and $H(x,y)=y$ on the line $x=y$. Then we can conclude that
this height function is a limiting height function for the domain bounded by the $x$-axis and the line $x=y$. 
In this case the white region outside the red curve is a ``neutral'' region where there is (conjecturally)
no limit shape and the height is random.

\subsubsection{Boundary curves}
Since $F_p=-1$ and higher order derivatives of $F$ are zero, the parametrisation for the boundary curves can be readily read off from (\ref{xinF}):

\begin{itemize}[itemsep = 4pt]
 \item \textbf{Green boundary}: $p<0$
 $$(x,y) = \left( -\frac{(r^2-2p)(1-p)^2}{r^2(p^2-2p+r^2)}, \frac{p^2 \left(p-r^2\right)^2}{r^4 \left(p^2-2 p+r^2\right)} \right)
$$ 
Below and on this boundary the limit shape has a facet where the height is constant $H=0$.

\item \textbf{Blue boundary}: $0<p<r^2$
 $$(x,y) = \left( -\frac{1-2 p}{r^2} , \frac{p^2}{r^2} \right)
$$ 
The height function is $C^1$ but not analytic across this curve ($(s,t)$ tends to $(0,1)$ in the gray region as one approaches the blue/gray boundary). Along this curve the height varies linearly, $H(x,y)=y$, corresponding to the height on the black line. The blue curve is the boundary of the grey region until it intersects the red curve where it bends away from the grey region. We don't know the actual boundary between the blue facet and the white
neutral region, except that the blue facet necessarily contains the region directly to its right, as illustrated.

\item \textbf{Red boundary}: $r^2<p<1$
 $$(x,y) = \left( -\frac{r^2-3 p+1}{r^2} , \frac{p^3}{r^4} \right)
$$ 
This curve bounds the grey region and the height function is not analytic across this curve. To the right of the red curve the height is neither constant nor linear, but apparently random.

\item \textbf{Purple boundary}: $1<p$
 $$(x,y) = \left( -1+\frac{2 p}{r^2} , \frac{p^2}{r^4}  \right)
$$ 
The purple boundary is analogous to the blue boundary. The height function is $C^1$ but not analytic across this curve, and along the curve it varies linearly, $H(x,y)=x$. The purple facet contains the region directly below the purple curve,
and possibly some part to the left of this region as well.
\end{itemize}

\subsubsection{Boxed plane partition}\label{bppsection}

For the ``boxed plane partition'' region of Figure \ref{bppr}
with vertices $\{((-1,-1),(0,-1),(1,0),(1,1),(0,1),(-1,0)\}$
we take
\begin{equation}
  \label{gbpp}
g(u) = \left(1-\frac{r}{u}\right)\sqrt{u^2+\frac{-3 + 2 r - 3
    r^2}4u+r^2}.
\end{equation}
See Figure \ref{bppexact}. This formula is valid for $r\in(1/3,1)$. It has facets analogous to the case $F=u$ above (not drawn). 
The oval repulsive region degenerates, as $r$ decreases to $1/3$, to a segment. See Figure \ref{bppbdy}.
(For an asymmetric hexagon, with side lengths $n_1,n_2,n_3,n_1,n_2,n_3,$ 
this degeneration will take place at a different value of $r$ depending on the ratios $n_1/n_2, n_2/n_3$). 
\begin{figure}[htbp]
\center{\includegraphics[width=3in]{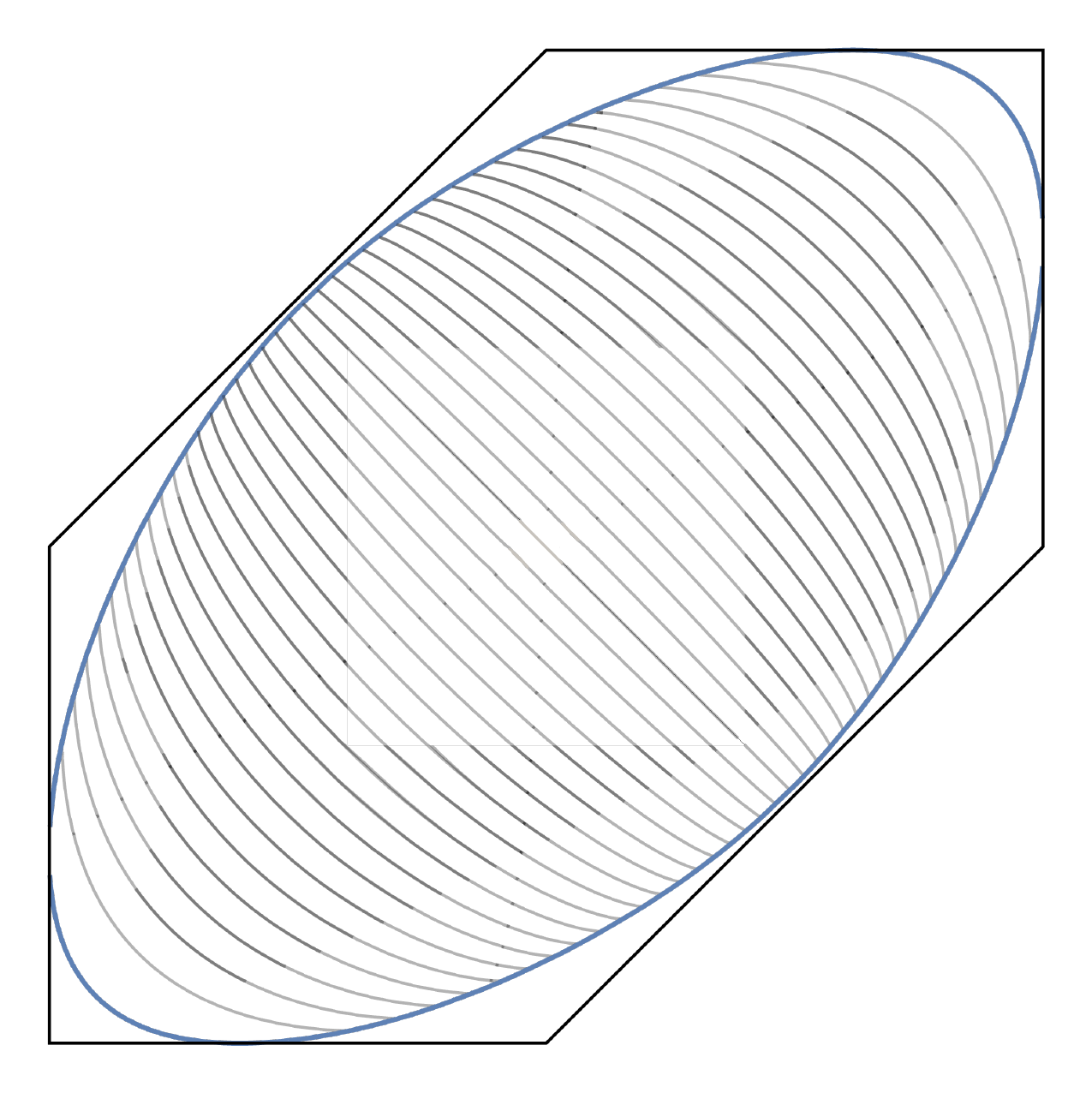}}
\caption{\label{bppexact}The boxed plane partition limit shape, shown with height contours in the repulsive region.
Here $r=0.7$.}
\end{figure}
\begin{figure}[htbp]
\center{\includegraphics[width=1.2in]{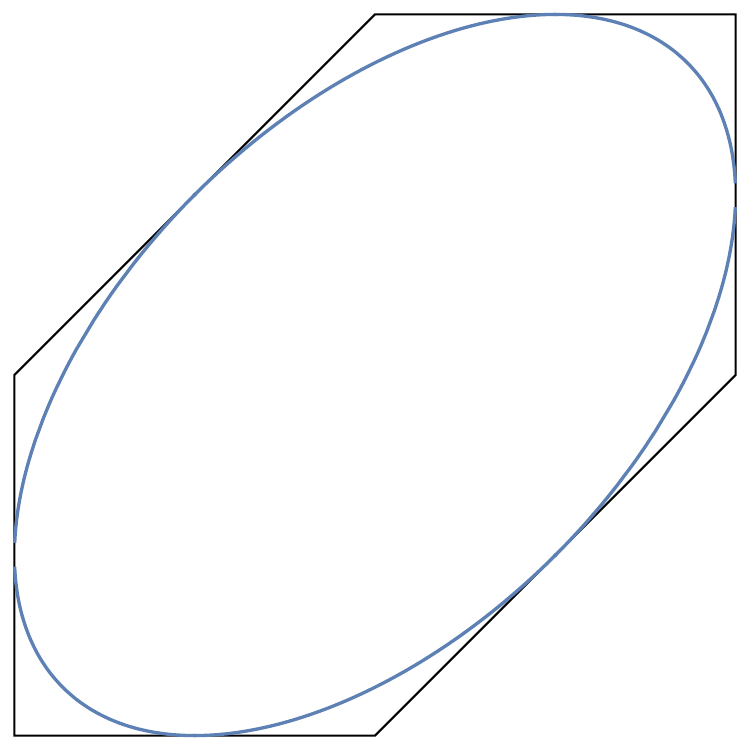}\includegraphics[width=1.2in]{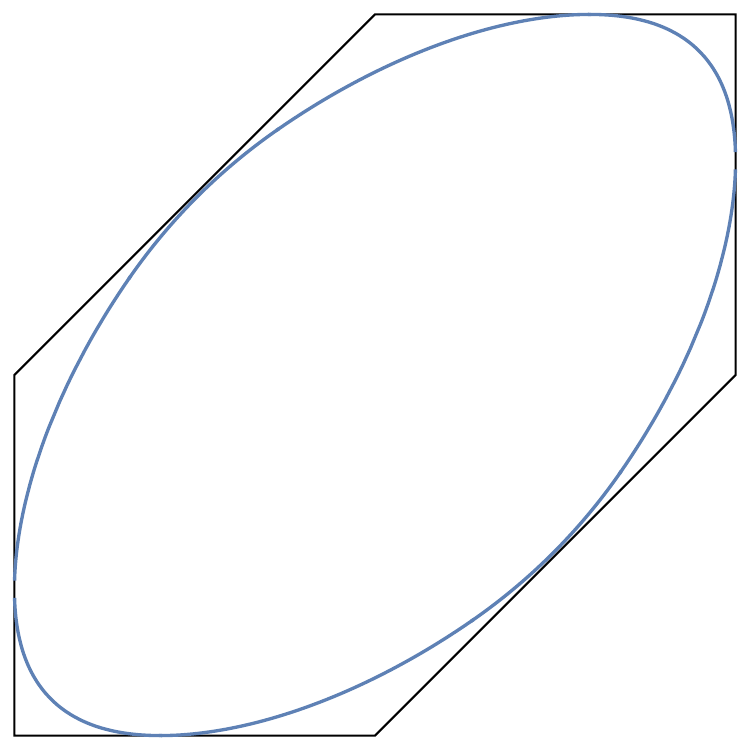}\includegraphics[width=1.2in]{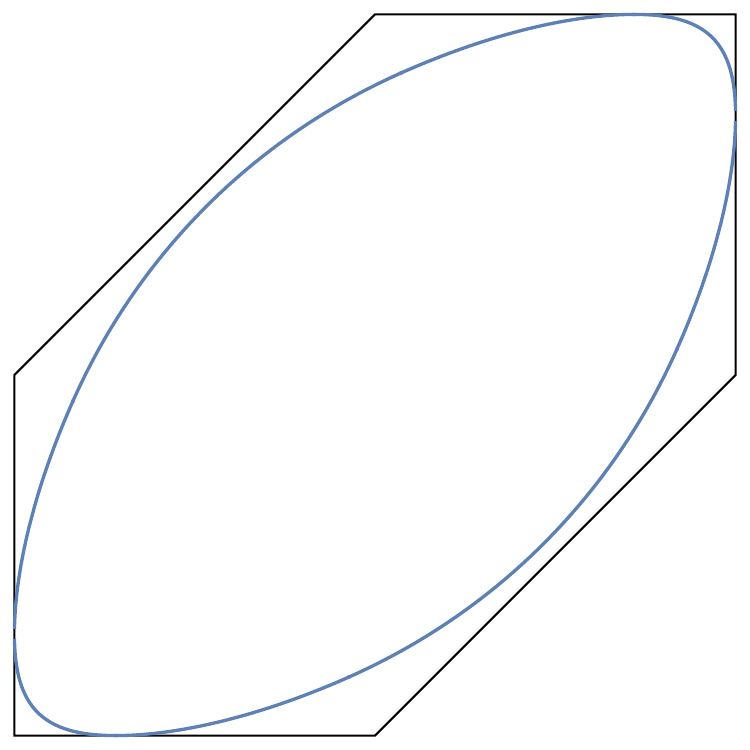}\includegraphics[width=1.2in]{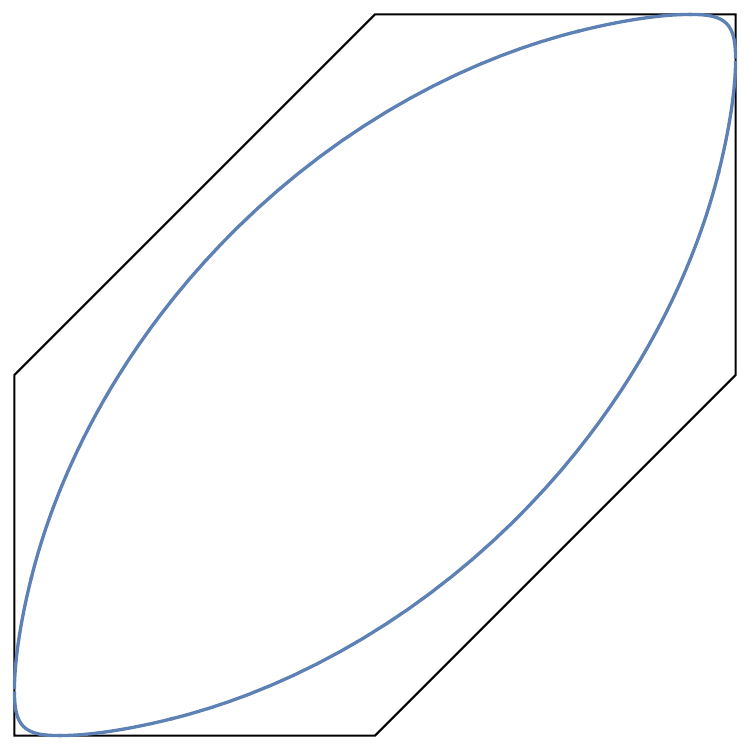}\includegraphics[width=1.2in]{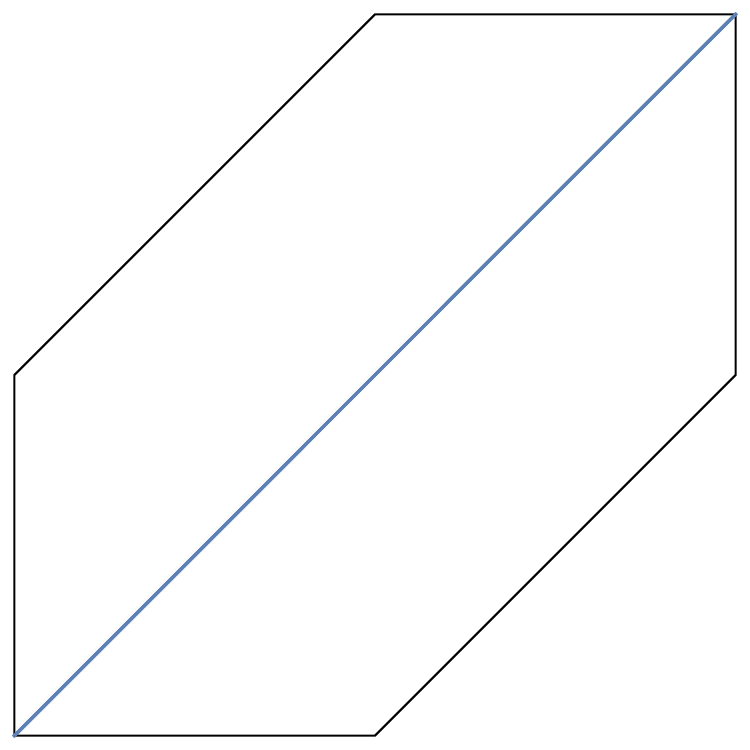}}
\caption{\label{bppbdy}The boxed plane partition repulsive region boundary, for $r=1,\frac56,\frac23,\frac12,\frac13$.
At $r=1/3$ the repulsive region has collapsed. For $r<1/3$ there is no repulsive region.}
\end{figure}

We can argue that this is the correct repulsive region as follows. First note that if we change the definition of
$\sigma$ by adding
to it a linear function of $s$ and $t$, we don't change any limit shape: limit shapes depend on the second derivatives of $\sigma$. It is convenient to add $(s+t)\log(1-r^2)$
so that $\sigma$ is \emph{minimized on the neutral region $\N\setminus\N^*$}. Then we immediately see that
if there is a height function whose gradient lies entirely in the neutral region $\N\setminus\N^*$,
then any such function is a minimizer, and all minimizers will have this form. This is the case for $r<1/3$ as illustrated in Figure \ref{bppPL}: the piecewise linear solution in that figure has gradient in the neutral
region for any $r\le 1/3$ (note that $r=1/3$ is the largest $r$ for which $(1/4,1/4)$ is in the neutral region). 
\begin{figure}[htbp]
\center{\includegraphics[width=2in]{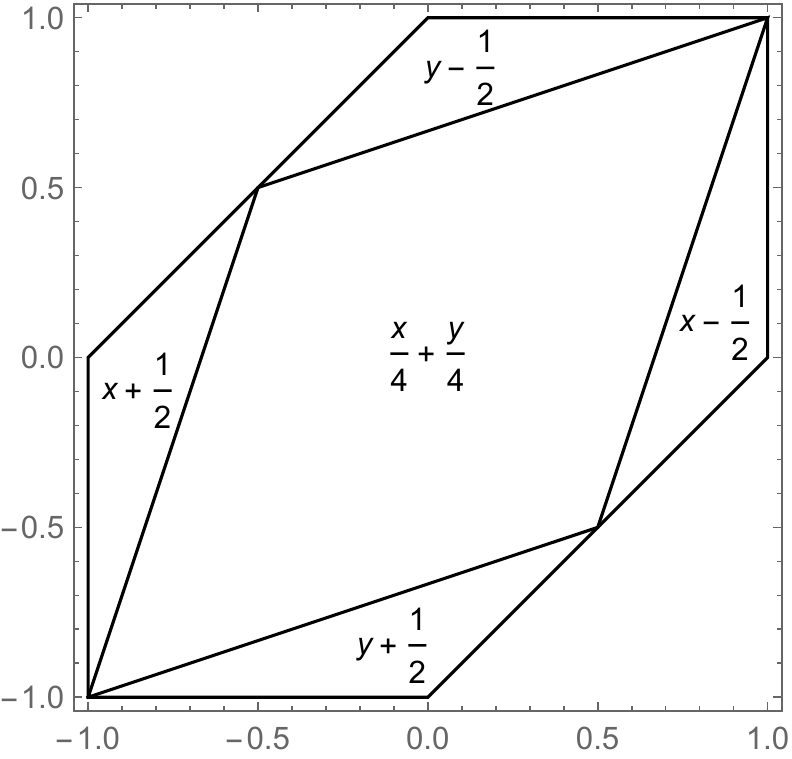}}
\caption{\label{bppPL}A PL minimizer when $r\le 1/3$.}
\end{figure}

For $r>1/3$ there must be a nonempty repulsive region: this follows 
because along the central line from $(-1,-1,-1/2)$ to $(1,1,1/2)$, where we have $s=t$ by symmetry,
there must be a point of slope $s+t=2s\le \frac12$, and such a slope $(s,s)$ is in $\N^*$.
The repulsive region has the same four-fold symmetry as the boundary values,
and each slope $(s,t)\in\N^*$
occurs at two centrally symmetric points (except at the center which is a point of multiplicity two). 
The repulsive region 
must touch the left and right boundaries at $x=\pm1$ when $u=\infty$ and the upper and lower boundaries
($y=\pm1$) when $u=0$. 
Thus $g$ is a double branched cover of $\H$, branched over a certain point $a(r)$ which maps to the center. 
From (\ref{gprime}) and symmetry we see that $g$ above is of the form $q(u)\sqrt{u^2+a(r)u+r^2}$ where $q$ is rational; 
now a somewhat technical argument shows that $q$ and $a(r)$ are uniquely determined to be the above values by the equations (\ref{xinF}) and the fact that $x,y$ remain bounded as $u$ runs over $\R\cup\{\infty\}$.

\section{$r>1$ case}\label{bigr}

\subsection{Maximal eigenvalue: the $r>1$ case}

Recall that for $r>1$ the relevant Bethe roots $w_1,\dots,w_n$ are the $n$ roots of $w^{N-n}(1-w)^n=y$
of smallest real part.
Let $C_\rho$ be the circle centered at $0$ of the 
appropriate radius $\rho$ chosen to enclose $w_1,\dots,w_n$ and no other roots.
If there are two components to the Cassini oval, the circle $C_\rho$ separates the two components only if $N=2n$.
If there is one component, or if $N\ne 2n$, $C_{\rho}$ will hit the Cassini oval at exactly two points.

\subsubsection{Case $N=2n$ and two components}

In case $N=2n$, that is, $s=1/2$, and there are two components,
$C_{\rho}$ will have one component in its inside and one in its outside. 

By Lemma \ref{lem:Mahler}, the integral of $\log|p(w)-y|$ around $C_\rho$ is the sum of the log moduli 
of the roots \emph{outside}
$C_{\rho}$, plus $n\log\rho$. The summation in \eqref{eq:cdef} for the roots inside is obtained by subtracting the sum for the outside roots
from the sum for all the roots, that is, from $\log|y|.$ 

Thus \eqref{eq:cdef} gives
\begin{equation}
\label{bigr1}-\log((r^2-1)\e^{Y}) = \frac1n\left(\log|y| + n\log\rho - \frac1{2\pi}\int_{0}^{2\pi}\log|p(\rho \e^{i\theta})-y|\,d\theta \right).
\end{equation}

In the integral we can ignore the $y$ term up to negligible errors, leading to 
$$-\log((1-r^2)\e^Y)=\frac1{n}\log|y|.$$

For the eigenvalue  \eqref{eq:eval3} we need to compute 
\begin{equation}
  \prod_{i=1}^n 1+(r^2-1)w_i.
\end{equation}
 
The map $w\mapsto 1+(r^2-1)w$ takes points of the oval inside of $C_\rho$ to points inside
the disk $C_R$ at $0$ of radius $R=|1+(r^2-1)w_0|$ (and points outside go to points outside).

We use the same equation as before, (\ref{dp2}), except that (\ref{mid}) gives
$$\log\prod_{i=1}^n 1+(r^2-1)w_i = \log|c_0|-n\log r^2.$$

This leads to 
\begin{equation}
\label{fsybig1} F_m\left(\frac12,Y\right) = \frac{Y}{2}.
\end{equation}

\subsubsection{Case $N\ne 2n$ or one component}\label{onecpt}

In this case 
let $w_0\in\C$ (with $\rho=|w_0|$) be the point in the upper half plane at which the circle of radius $\rho$ intersects the oval. It is convenient to use 
$w_0$ as a new variable. We have $y=-|p(w_0)|$ and we'll see how $r$ and $n/N$ are functions of $w_0$ as well. 
In the formula (\ref{bigr1}), for $w=\rho \e^{i\theta}$ outside the oval,
the integrand is (up to terms tending to zero as $N\to\infty$) $\log|p(w)|$ and for $w$ inside the oval, the integrand is $\log|y|$.

We get up to negligible errors for $N$ large
$$-\log(|1-r^2|\e^{Y})= \log\rho+\frac1n\log|y|-\frac{\arg w_0}{n\pi}\log|y| - 
\frac1{2\pi n}\int_{\arg w_0}^{2\pi-\arg w_0}\log|p(\rho \e^{i\theta})| d\theta$$
and plugging in $p(w)=w^{N-n}(1-w)^n$ and $|y|=|p(w_0)|$ and simplifying  
gives
\begin{align}\label{rfromw2}
  -\log\left(|1-r^2|\e^{Y}\right)&=\log|w_0| + \left(1-\frac{\arg
    w_0}{\pi}\right)\log|1-w_0|-\frac1{\pi}\Im\Li(w_0)\\
    &=\log|w_0|+\log|1-w_0|-B(w_0).\nonumber
\end{align}

The relationship between $w_0$ and $s$ is as follows. 
We suppose $(N-n)/n=p/q$ in lowest terms with $p,q$ odd; the other parities lead to similar arguments. 
Consider the map $\psi(w)= w^{(N-n)/n}(1-w)$ with branch cut along the
positive $x$-axis, and with the branch chosen so that the negative real axis maps to itself. This $\psi$ maps the roots $w_1,\dots,w_n$ to the $n$th roots of $y$, with consecutive roots 
(if the Cassini oval has one component) mapping to consecutive $n$th roots. 
If the Cassini oval has two components, we order the roots $w_j$ in a slightly unconventional way: those in the upper half plane 
are ordered with decreasing real part, and those in the lower half plane are ordered with increasing real part. In this order 
$\psi$ maps the $w_i$ to consecutive $n$th roots of $y$ 
(in particular the leftmost root in $\H$ of the right oval and right most root in $\H$ of the left oval
map to consecutive $n$th roots of $y$). 
By these choices, $\psi(w_0)=|y|^{1/n}e^{\pi i/n}$ and the $\psi$-images of the 
roots wrap exactly once around the circle of radius $|y|^{1/n}$,
regardless of whether the Cassini oval has one or two components.
Thus for the arguments, 
$$-\frac{N-n}{n}(\pi-\arg w_0)+\arg(1-w_0)=\pi/n,$$ 
which implies (for $n$ large) that the triangle with vertices $0,1,w_0$ has \emph{exterior} angles in ratio $(1-s):s$,
that is, if $\theta_1,\theta_2$ are the angles of the triangle at $0,1$ respectively
then $(1-s)(\pi-\theta_1) = s(\pi-\theta_2)$. We define $\theta$ so that these exterior angles are $s\theta$ and $(1-s)\theta$,
see Figure \ref{wbigr}.
\begin{figure}[htbp]
\center{\includegraphics[width=2in]{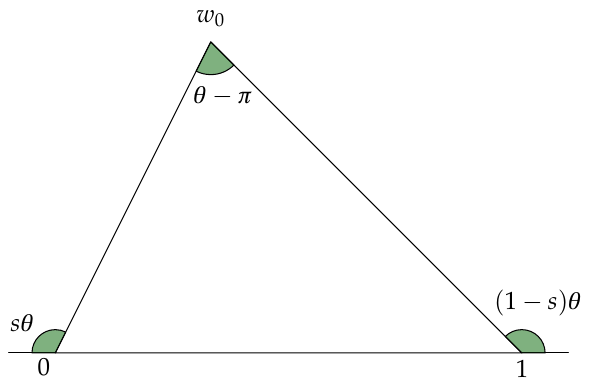}}
\caption{\label{wbigr}Relationship between $w_0,s,\theta$ in the $r>1$ case. }
\end{figure}

This leads to 
$$s=\frac{\pi-\arg w_0}{2\pi-\arg(\frac{w_0}{1-w_0})}.$$

\begin{figure}[htbp]
\center{\includegraphics[width=2.5in]{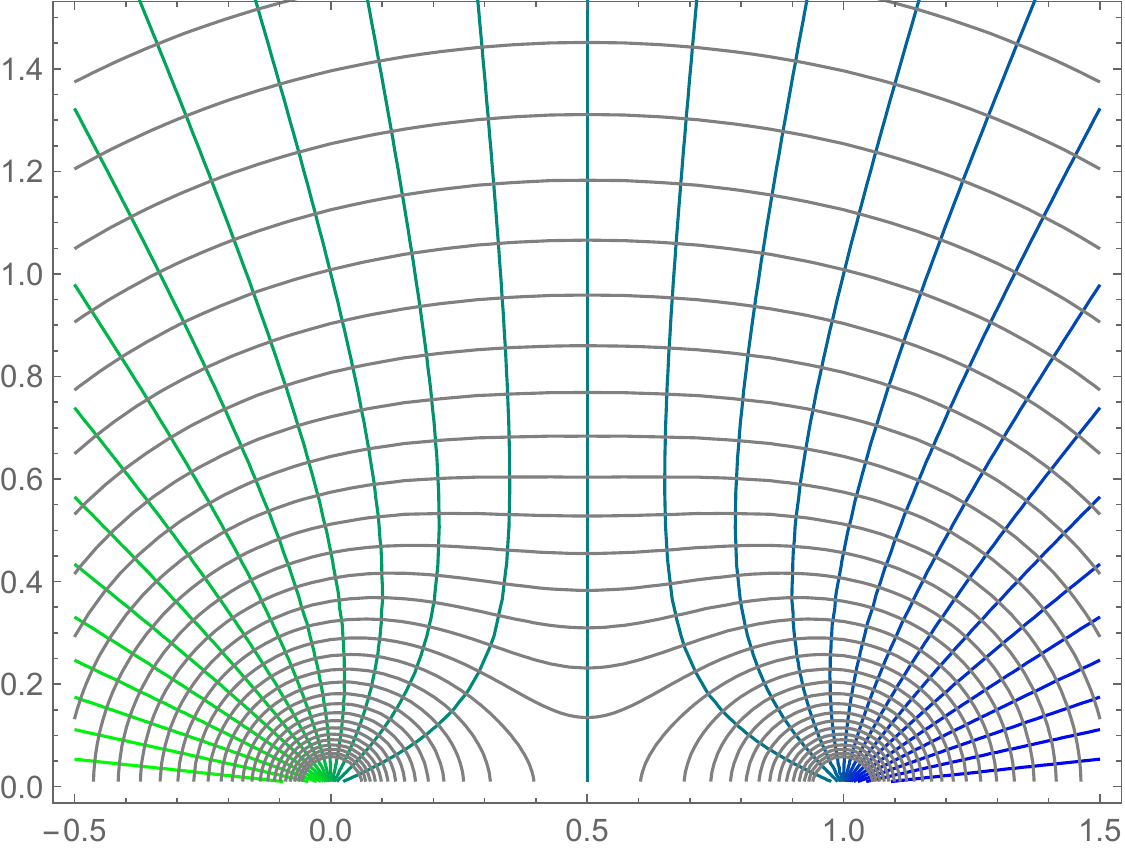}}
\caption{\label{Ywbigr}In blue/green, the level curves of $s=\frac{\pi-\arg w_0}{2\pi-\arg(\frac{w_0}{1-w_0})}$. The orthogonal family of level curves of $Y$ is shown in gray. }
\end{figure}

The monotonicity of $Y$ as a function of $y$ in the limit of large $N$ 
follows from the orthogonality of the level curves of $s$ and those of $Y$, see
Figure \ref{Ywbigr}. For large finite $N$ monotonicity follows by an argument similar to that of Lemma \ref{mono}.

We now compute the eigenvalue \eqref{eq:eval3}.
For this we need to compute 
\begin{equation}\prod_{i=1}^n 1+(r^2-1)w_i.\end{equation}
 
The map $w\mapsto 1+(r^2-1)w$ takes points of the oval inside of $C_\rho$ to points inside
the disk $C_R$ at $0$ of radius $R=|1+(r^2-1)w_0|$ (and points outside go to points outside).

At this point the computation is identical to that for the $r<1$ case, except that $\overline{w}_0$ there is $w_0$ here.
With $u_0=1+(r^2-1)w_0$, the microcanonical free energy is 
\begin{equation} F_m(s,Y)= (1-s)Y+(1-s)B(u_0)+sB\left(1-\frac{r^2}{u_0}\right).
\end{equation}

\subsection{Legendre transform, $r>1$ case}

Using $Y=-\log(r^2-1)+B(w_0)-\log|w_0(1-w_0)|,$ we compute using (\ref{dB})
\begin{align*}\frac{dY}{d\theta} &= \frac{dB(w_0)}{d\theta} - \frac{d}{d\theta}\log|w_0(1-w_0)|\\
&=\left(1-\frac{s\theta}{\pi}\right)\frac{d}{d\theta}\log|1-w_0| + \left(1-\frac{(1-s)\theta}{\pi}\right) \frac{d}{d\theta}\log|w_0| - \frac{d}{d\theta}\log|w_0(1-w_0)|\\
&=-\frac{s\theta}{\pi}\frac{d}{d\theta}\log|1-w_0| -\frac{(1-s)\theta}{\pi}\frac{d}{d\theta}\log|w_0|.
\end{align*}

We have
\begin{align*}\frac{dF_m(s,Y)}{d\theta} &= (1-s)\frac{dY}{d\theta}  +(1-s)\frac{dB(u_0)}{d\theta}+s\frac{dB(1-\frac{r^2}{u_0})}{d\theta}\\
&= (1-s)\frac{dY}{d\theta}  +(1-s)\phi\frac{d}{d\theta}\log|1-u_0| +s\phi\frac{d}{d\theta}\log|u_0-r^2|\\
&= (1-s)\frac{dY}{d\theta}  +(1-s)\phi\frac{d}{d\theta}\log\frac{|1-u_0|}{r^2-1} +s\phi\frac{d}{d\theta}\log\frac{|u_0-r^2|}{r^2-1}\\
&= (1-s)\frac{dY}{d\theta}  +\phi\left((1-s)\frac{d}{d\theta}\log|w_0| +s\phi\frac{d}{d\theta}\log|1-w_0|\right)\\
&= \left(1-s-\frac{\phi}{\theta}\right)\frac{dY}{d\theta}.\\
\end{align*}

Thus
$$t=\frac{dF_m(s,Y)}{dY} = 1-s-\frac{\phi}{\theta}$$
and $\phi=(1-s-t)\theta$
as in the $r<1$ case. See Figure \ref{rbigtri}.

\begin{figure}[htbp]
\center{\includegraphics[width=4in]{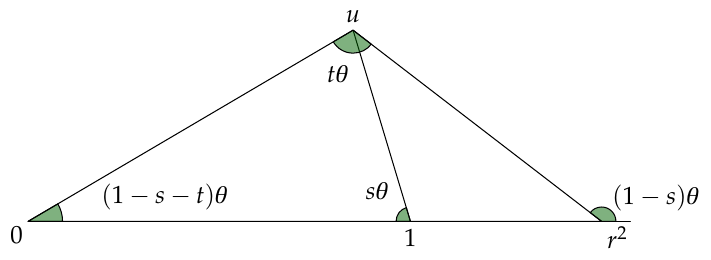}}
\caption{\label{rbigtri}Relationship between $u,s,t,\theta$ in the $r>1$ case.}
\end{figure}

The surface tension is then 
\begin{align}\sigma(s,t) &=-F_m(s,Y)+Yt \nonumber\\
&=  -(1-s-t)Y-(1-s)B(u_0)-sB\left(1-\frac{r^2}{u_0}\right)\nonumber\\
&=(1-s-t)(\log(r^2-1)-B(w_0)+\log|w_0(1-w_0)|)-(1-s)B(u_0)-sB\left(1-\frac{r^2}{u_0}\right).\label{stensionbig}
\end{align}

The surface tension and free energy are shown in Figure \ref{rbigsig}

\begin{figure}[htbp]
\center{\includegraphics[width=3in]{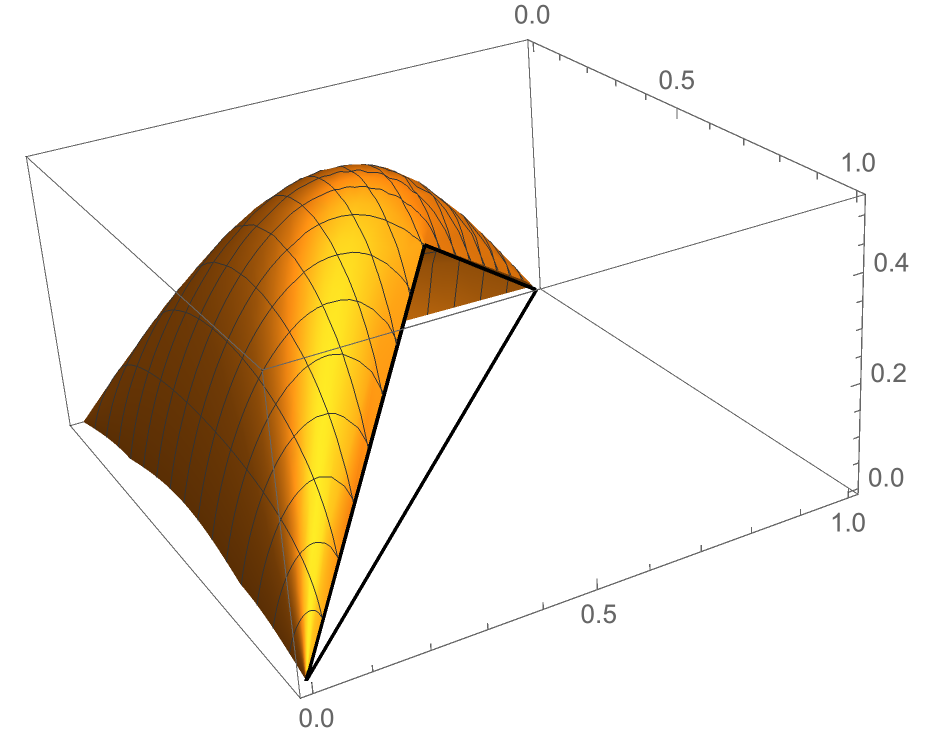}\includegraphics[width=3in]{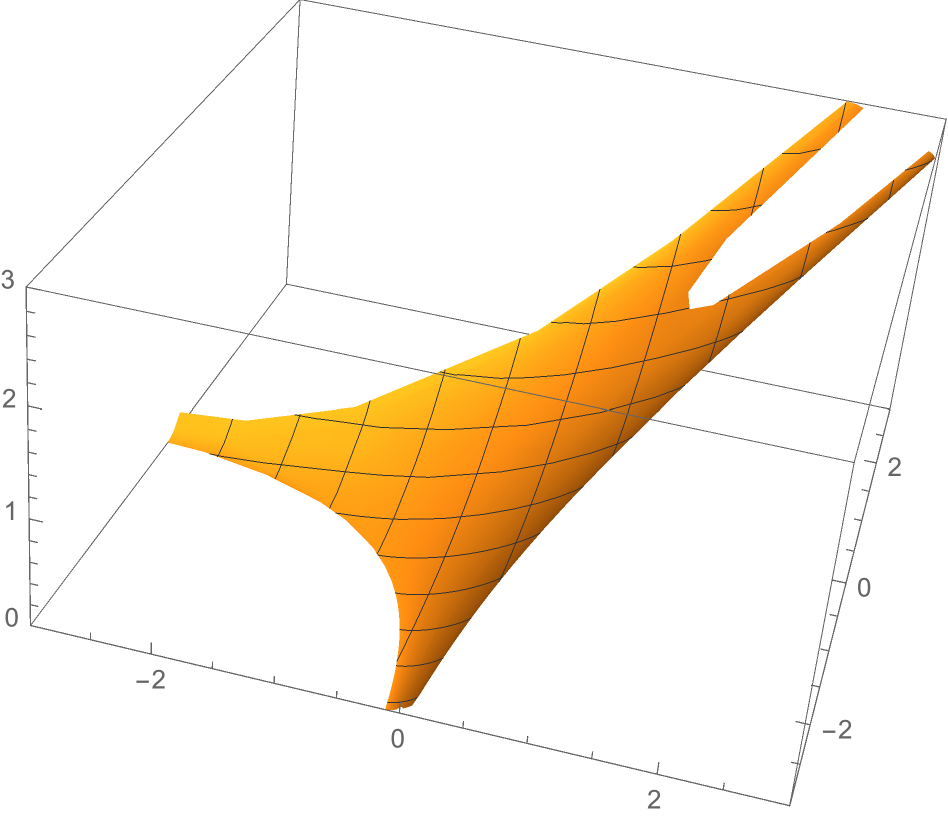}}
\caption{\label{rbigsig}Minus the surface tension and free energy in case $r=1.3$. The surface tension is analytic on $\N$, zero on $s=0$ and $t=0$
and equal to $\min\{1-s,1-t\}2\log r$ when $s+t=1$. The free energy is linear on the complementary components of
the curved part illustrated.}
\end{figure}

\subsection{Euler-Lagrange equation, $r>1$ case}

The $r>1$ case is quite similar to the $r<1$ case.
We define $z$ where $\bar z$ is the analog of $w$ when the roles of $s$ and $t$ are reversed.
Then, in terms of $u$, we have
\be \label{wzfromubigr} w=\frac{u-1}{r^2-1},\qquad z = \frac{r^2/u-1}{r^2-1}\ee
and 
\begin{equation}
\label{spcurvebig}1-w-z-(r^2-1)wz=0.\end{equation}
Note that this is the \emph{same relation} as in the $r<1$ case (Lemma \ref{spcurvesmall}).

We have 
\begin{align}X &= -\log(r^2-1)+B(\bar z)-\log|z(1-z)|\notag\\
Y &= -\log(r^2-1)+B(w)-\log|w(1-w)|\label{XYbigr}
\end{align}
and again (\ref{EL1}) and (\ref{EL2}) combine to give
\begin{prop}
 \label{prop:ELrbig}
For $r>1$ the EL equation (\ref{EL1}), combined with the mixed partials equation (\ref{EL2}), reduces to 
\begin{equation}
\label{EL3rbig}\left(-B(z)_z-\frac{1-2z}{2z(1-z)}\right)z_x+\left(B(w)_w-\frac{1-2w}{2w(1-w)}\right)w_y=0.\end{equation}
\end{prop}

Again using $\frac{dw}{w(1-w)}=-\frac{dz}{z(1-z)}$ and $w_xx_{\bar w}+w_yy_{\bar w}=0$ this can be written
\begin{lemma}
Let $r>1$.
The $(x,y)$ coordinates of the limit shape satisfy the linear first-order PDE
\begin{equation}\label{argentbigr}(\A(w)-(1-2w))x_{\overline w} - (\A(z)+(1-2z))y_{\overline w}=0.\end{equation}
\end{lemma}

\subsection{Integration}

Equation (\ref{argentbigr}) for $r>1$ is
$$(-w(\arg w-\pi)-(1-w)(\arg(1-w)+\pi)x_{\overline w}+(z(\arg z+\pi)+(1-z)(\arg(1-z)-\pi)y_{\overline w}=0$$
and dividing by $\theta,$
$$(ws-(1-w)(1-s))x_{\overline w} + (zt-(1-z)(1-t))y_{\overline w}  =0,$$
so we get the same equation (\ref{deriv}) as in the $r<1$ case.

As in that case the equation (\ref{deriv}) becomes
$$(w-1)x_{\overline w}+ (z-1)y_{\overline w} + H_{\overline w}=0,$$ which we can integrate to get
\begin{theorem} \label{univlarger} For $r\ne 1$ the following equation holds:
\begin{equation}\label{Heq2} (w-1)x+(z-1)y+H + f(w) = 0\end{equation}
where $f$ is an analytic function depending only on boundary conditions.
\end{theorem}
 
With $u=\RR \e^{i\phi}$ taking the imaginary part of (\ref{Heq2}) gives (absorbing the $(r^2-1)$ factor into $f$)
\begin{equation}\label{yfromx2}x-\frac{r^2}{\RR^2}y = \frac{g(u)-\overline{g(u)}}{u-\overline u}\end{equation}
where $g(u)=-(r^2-1)f(w)$. For any given $g$ this can be solved for $y$, and, plugging back into (\ref{argentbigr}),
gives an equation of the form $$A_1x_u+ A_2x =A_3$$ where
$A_1,A_2,A_3$ are functions of $u$. This can be integrated by standard techniques. 

As an example, take $g\equiv 0$. Then
$$(\A(w)-(1-2w))x_{\bar u} - (\A(z)+(1-2z))\left(x\frac{\RR^2}{r^2}\right)_{\bar u}=0.$$
Rearranging gives
$$\frac{x_{\bar u}}{x} = \frac{(\A(z)+1-2z)u}{r^2(\A(w)+2w-1)-u\bar u(\A(z)+1-2z)}$$
and using Lemma \ref{Ader}, we note that the numerator is minus the $\bar u$-derivative of the denominator, so
$$x=\frac{C}{r^2(\A(w)+2w-1)-u\bar u(\A(z)+1-2z)}$$
for a real constant $C$. 
See Figure \ref{bigrfund}.
\begin{figure}[htbp]
\center{\includegraphics[width=3in]{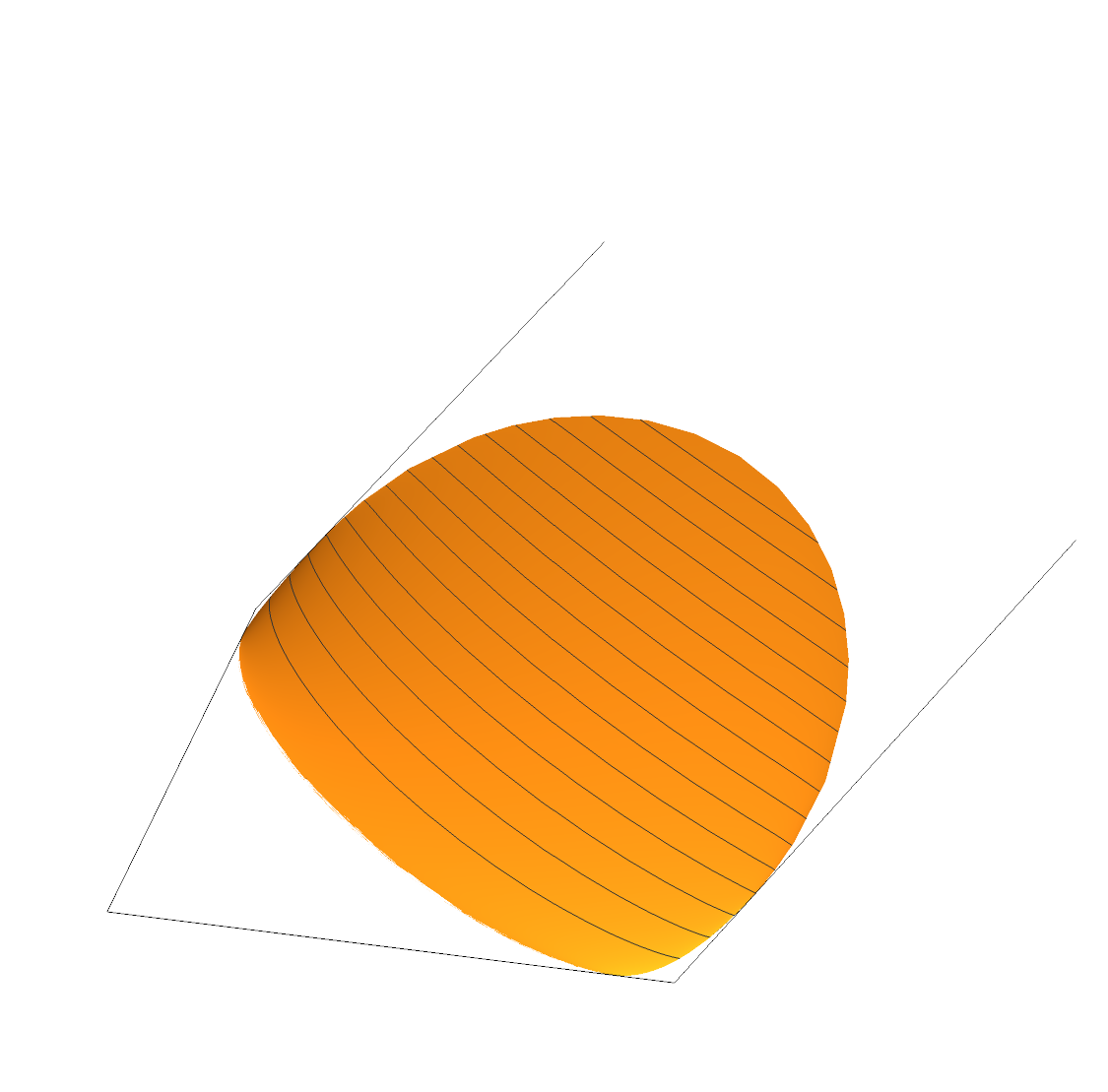}}
\caption{\label{bigrfund}The ``thumbnail": the $g\equiv 0$ limit shape for $r=3.1$. (It is also one of the 
two liquid region
limit shapes for the boxed plane partition when the diagonal side is sufficiently long: see Figure \ref{bppbigr} below).}
\end{figure}

The analog of equation (\ref{x}) for $r>1$ for general $g$ is
$$x = \frac1{r^2(\A(w)-1+2w)-\RR^2(\A(z)+1-2z)}\left[
\Im\left(\frac{1}{\pi}\int \frac{r^2g(u)}{(r^2-u)(u-1)}du\right)+
(\A(z)+1-2z)\RR^2k(u)-\frac{g(u)r^2\theta}{\pi(r^2-1)}\right].$$

The arctic boundary is piecewise analytic as before; the formulas for $p\in\R$ are 
\begin{equation}\label{xinF2}
x=\begin{cases}
\displaystyle F_p\frac{(r^2-2p)(1-p)^2}{r^2(p^2-2p+r^2)}+F_{pp}\frac{p(r^2-p)(1-p)^2}{r^2(p^2-2p+r^2)} 
- \frac{c_1(1-r^2)}{\pi r^2(p^2-2p+r^2)}&p<0\\[.15in]
\displaystyle F_p\frac{1-2p}{r^2}+F_{pp}\frac{p(1-p)}{r^2}- \frac{c_2(1-r^2)}{\pi(p-r^2)^2}&0<p<1\\[.15in]
{\scriptstyle F_p\frac{r^4 p^2-4 r^4 p+2 r^4-2 r^2 p^3+8 r^2 p^2-4 r^2 p-2 p^3+p^2}{r^2 \left(r^4+r^2
   p^2-4 r^2 p+r^2+p^2\right))}+F_{pp}\frac{(p-1) p(r^2-p)(p r^2+p-2 r^2)}{r^2(p^2 r^2+p^2-4 p
   r^2+r^4+r^2)} + \frac{2c_3(r^2-1)}{\pi(p^2 r^2+p^2-4 p r^2+r^4+r^2)}}&1<p<r^2\\[.15in]
\displaystyle F_p\frac{r^2-2p}{r^2}+F_{pp}\frac{p(r^2-p)}{r^2}- \frac{c_4(1-r^2)}{\pi r^2(p-1)^2}&r^2<p\end{cases}
\end{equation}
for constants $c_1,c_2,c_3,c_4$ which depend on the height function on each facet.
The first, second and fourth terms are identical to that in the $r<1$ case. All four of these can be obtained
directly from (\ref{Heq2}) and (\ref{yfromx2}) by solving for $x$ and $y$.

Figure \ref{bppbigr} shows the arctic boundary for the boxed plane
partition. For $1\le r\le 3$,  $g$ is given by (\ref{gbpp}); this is the
same as the expression for $g$ (but with different $r$) when $\tfrac13<r<1$. For $1\le r\le 3,$ there is
one interior component, which breaks up for $r>3$ into two components, each of which is equivalent up to isometry 
to the $g=0$ case.
\begin{figure}[htbp]
\center{\includegraphics[width=1.7in]{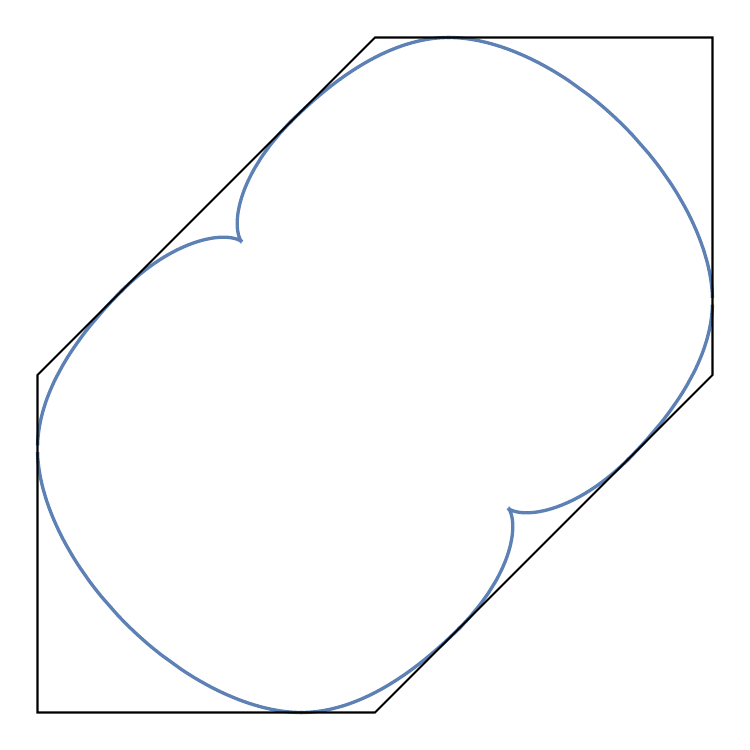}\includegraphics[width=1.7in]{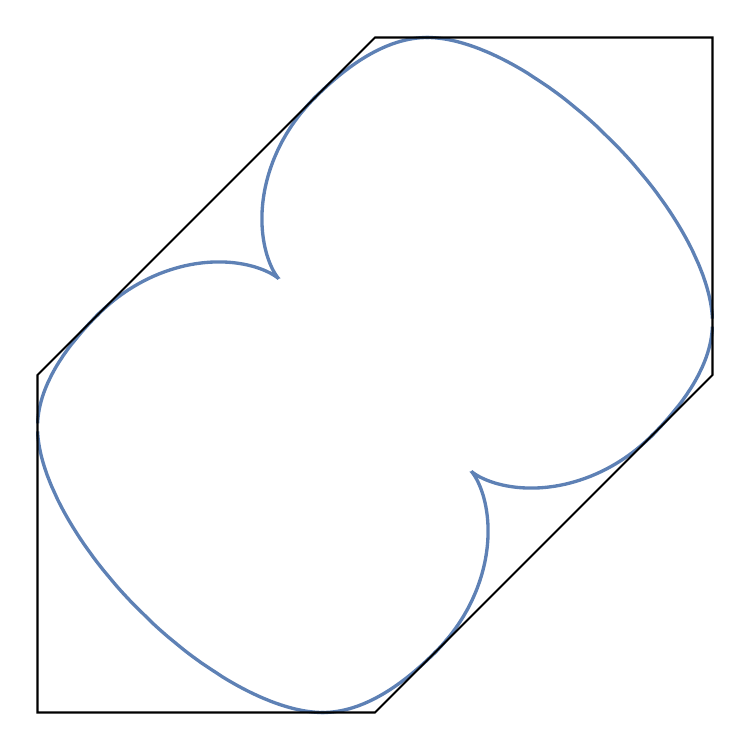}\includegraphics[width=1.7in]{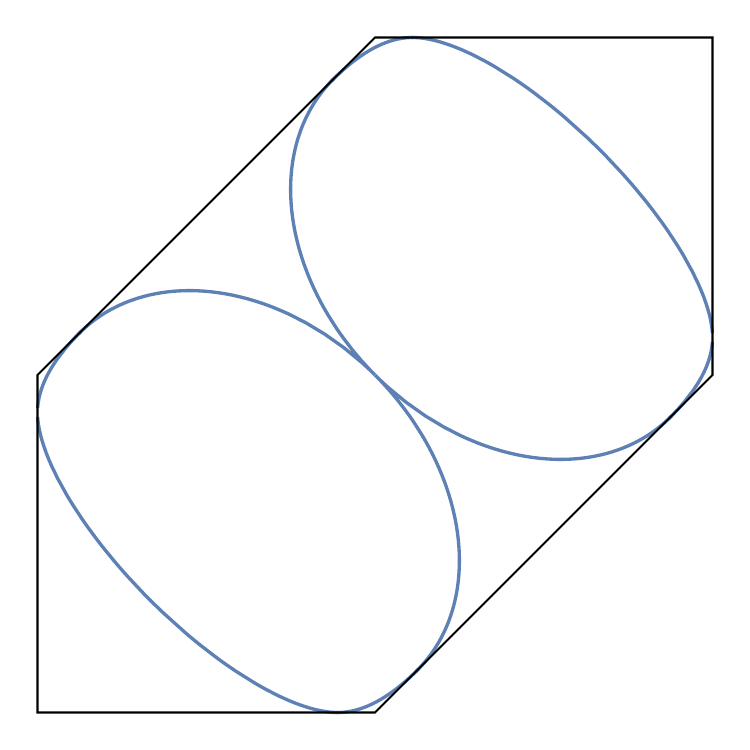}\includegraphics[width=1.7in]{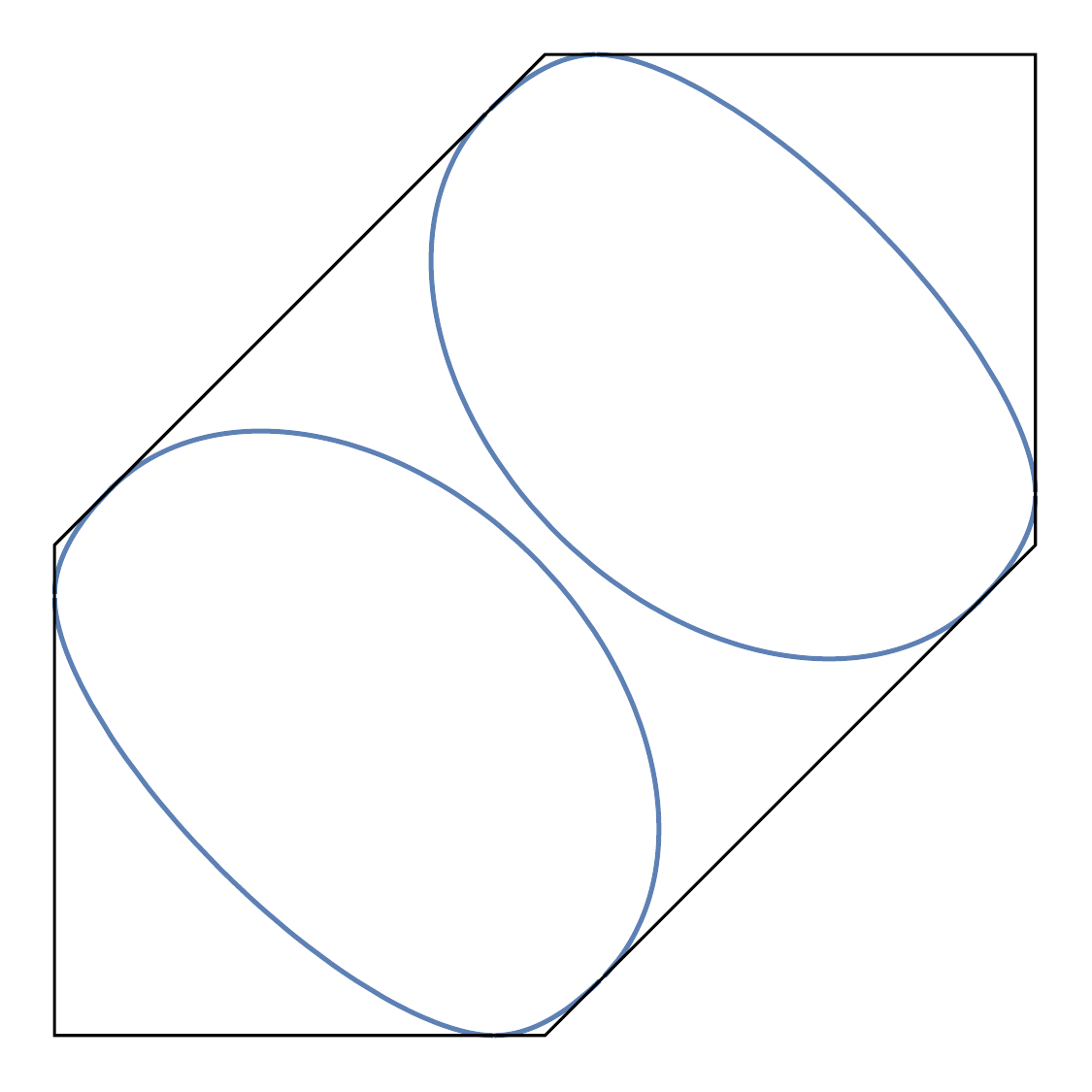}}
\caption{\label{bppbigr}The arctic boundary for the boxed plane partition in the cases $r=2,2.5,3,3.1$.}
\end{figure}

\begin{figure}[htbp]
  \center{\includegraphics[width=3.2in]{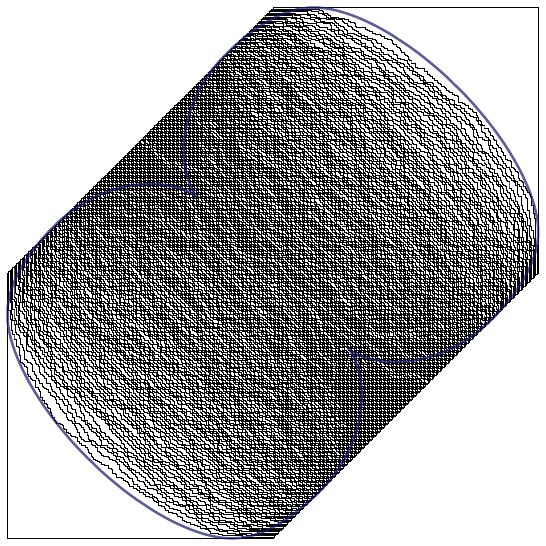}
          \includegraphics[width=3.2in]{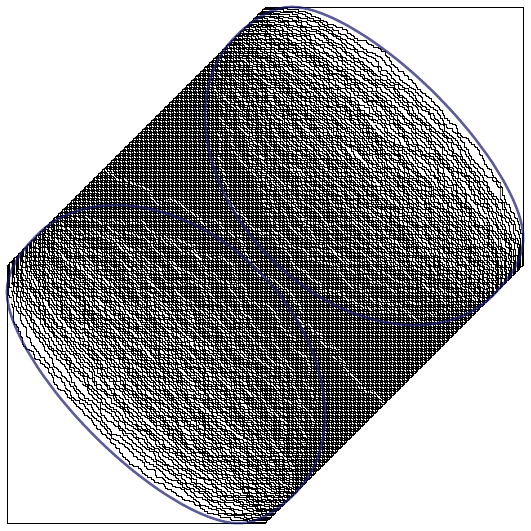}}
\caption{\label{bppbigrsim2}Simulations with $n=400$ and $r = 2.5$
  (left) and $r = 3.1$ (right), with the theoretical arctic boundary
  superimposed.}
\end{figure}

\section{Leading eigenvalue}\label{perron}
Here we explain why the eigenvalue we computed is the leading eigenvalue. This follows
by analytic continuation from the $r=1$ case.

When $r=1$, the unnormalized Bethe roots $\zeta_i$ are distinct $N$th roots of $1$, and the corresponding eigenvectors are ``antisymmetrized plane waves''
$$f_{\zeta_1,\dots,\zeta_n}(x_1,\dots,x_n) = \sum_{\sigma\in S_n} (-1)^{\sigma}\zeta_{\sigma(1)}^{x_1}\dots \zeta_{\sigma(n)}^{x_n}.$$ 
These form a complete eigenbasis: the exponentials $\zeta_1^{x_1}\dots \zeta_n^{x_n}$ are a complete
eigenbasis on the torus $(\Z/N\Z)^n$ (as the $\zeta_i$ take on all possible $N$th roots of unity), 
and the antisymmetrization operator $\A$ defined by
$$\A g(x_1,\dots, x_n) = \sum_{\sigma\in S_n}(-1)^{\sigma}g(x_{\sigma(1)},\dots,x_{\sigma(n)})$$
projects the Fourier basis to a basis for the antisymmetric functions on the torus $(\Z/N\Z)^n$, which we can identify
with the space of functions $\{f(x_1,\dots,x_n)\}$ with strictly increasing indices.

By the Perron-Frobenius theorem 
and strong connectivity of the state space,
the leading eigenvector at any value of $r$ is the unique eigenvector with
nonnegative entries, and its eigenvalue is the unique maximal eigenvalue. 
Our computed eigenvector and eigenvalue are analytic functions of the Bethe roots which
are themselves analytic in $r$
(and tend as $r\to1$ to the $N$th roots of $1$) in an interval $(r_0^-,r_0^+)$ around $1$. The only potential source of
nonanalyticity is when two real roots merge to become two complex roots, or vice versa. 
More precisely, this will lead to a nonanalyticity if a Bethe root merges with another root of $p(w)=y$ which is not a Bethe root. This indeed happens for $r<1$ at a certain value $r_0^-<1$, at the point where the Cassini oval breaks into two components, as illustrated in Figure \ref{noseplot}. 
For $r>1$  and $n=N/2$, this also happens at the point where the Cassini oval breaks into two components. For $r>1$ and $n\ne N/2$, this does not happen, so $r_0^+=\infty$ in these cases
(in these cases, there is indeed a point where two roots merge and separate, but they are either both Bethe roots
or neither is a Bethe root, and in either case the
expression for the Bethe eigenvalue and eigenvector--which are symmetric in the Bethe roots--are still analytic functions of these roots).

As we move $r$ continuously away from $1$, the leading eigenvalue never meets another eigenvalue,
and thus the eigenvector is the leading eigenvector throughout $r\in (r_0^-,r_0^+)$.

When $r$ decreases below $r_0^-$ or increases beyond $r_0^+$, the analytic continuation no longer necessarily holds. However in the $r<1$ case we do not in fact need the leading eigenvector for $r<r_0^-$; Lemma \ref{frozen} shows that at or beyond the point where the Cassini oval separates into two components,
the free energy in the large $N$ limit is $F(X,Y) = \max(X,Y)+\log(1-r^2)$. 
Likewise in the $r>1$ case with $n=N/2$, the free energy $F(X,Y)$ for $r\ge r_0^+$ is predetermined: 
the limiting free energy $F(X,Y)$ as $r\nearrow r_0^+$ tends to $F(X,Y) = X/2+Y/2+\log r$ at all points corresponding to $(s,t)=(1/2,1/2)$,
that is, at points $(X,Y)$ along the curve
$$(X,Y)=(-\log(r^2-1)-\log|z(1-z)|,-\log(r^2-1)-\log|w(1-w)|)$$ defining the boundary of relevant component of the ``amoeba'' complement, 
see equations (\ref{XYbigr}) and (\ref{wzfromubigr}). By convexity the free energy necessarily
equals $ (X+Y)/2 + \log r$ beyond this curve as well. As a consequence we don't need an expression for the leading eigenvector for $r>r_0^+$.

\section{Appendix}

\subsection{Mahler measure}

\begin{lemma}\label{lem:Mahler}
If $q(x)$ is a monic polynomial of degree $d$ with no roots of modulus $R$ then
$$\frac1{2\pi}\int_{0}^{2\pi} \log|q(R\e^{i\theta})| d\theta =  \sum_{|r_i|>R} \log|r_i| + \sum_{|r_i|<R} \log R$$
where the first sum is over roots $r_i$ of modulus greater than $R$ and the second sum is over the remaining roots (both with multiplicity).
\end{lemma}

\begin{proof}
  Since $z\mapsto \log |z|$ is harmonic away from the origin, the circle average    $\frac{1}{2\pi}\int_0^{2\pi}\log|a+b\e^{i\theta}|\,d\theta$ is equal to the value of $\log|z|$ at the center of the circle, namely $\log |a|$, as long as $|a|> |b|$. If $|a| < |b|$, then we can multiply by the argument of the logarithm by $1 = |\e^{-i\theta}|$ and then apply the same idea. So altogether, we have 
  \[
    \frac{1}{2\pi}\int_0^{2\pi}\log|a+b\e^{i\theta}|\,d\theta = \log(\max(|a|,|b|)). 
  \]
  The result then follows from summing over the roots of $q$. 
\end{proof}

\subsection{Dilogarithm}

For further information on the dilogarithm see \cite{zagier2007dilogarithm}.
The dilogarithm function $\Li(z)$ is defined by 
$$\Li(z) = -\int_0^z\log(1-\zeta)\,\frac{d\zeta}{\zeta}$$
for $z\not\in[1,\infty)$. 
From this formula we have the integral
\begin{equation}\label{backint}
\int_\theta^{2\pi-\theta}\log|1-r\e^{it}|\,dt = 2\Im\Li(r\e^{i\theta}).
\end{equation}

The Bloch-Wigner dilogarithm $D(z)$ is a variant of the dilogarithm defined by
$$D(z) = \arg(1-z)\log|z| + \Im\Li(z).$$
Here $D(z)$ is continuous on $\C$, real analytic except at $z=0,1$, and satisfies
the identities
$$D(z)=D\left(1-\frac1z\right)=D\left(\frac{1}{1-z}\right)=-D\left(\frac1z\right)=-D\left(1-z\right)=-D\left(\frac{-z}{1-z}\right).$$
In terms of the Lobachevsky function
$$L(\theta)=-\int_0^{\theta}\log(2\sin t)\,dt,$$ we have
$$D(z) = L(\alpha)+L(\beta)+L(\gamma)$$
where $\alpha,\beta,\gamma$ are the arguments of $z,1-\frac1z,\frac1{1-z}$ respectively, that is, the angles of the triangle
with vertices $0,1,z$. 

We also define 
\begin{align}\label{Btriangleform}B(z) &= \frac1{\pi}(D(z) + \alpha\log a+\beta\log b+\gamma\log c)\\
\nonumber&=\frac1{\pi}(\arg z\log|1-z|+\Im\Li(z))
\end{align}
where $a,b,c$ are the lengths of the edges of the $0,1,z$ triangle.

Note that for a 1-parameter family of triangles, from (\ref{Btriangleform}) we have
\begin{align*} 
  dB = \frac1{\pi}\Big(\log a\, d\alpha+\alpha\, d\log a&+\log b\, d\beta + \beta\, d\log b + \\ &\log c\, d\gamma+\gamma\, d\log c - \log(2\sin\alpha)\,d\alpha - \log(2\sin\beta)\,d\beta-\log(2\sin\gamma)\,d\gamma\Big), 
\end{align*}
which simplifies (using $d\alpha+d\beta+d\gamma=0$) to 

\begin{equation}\label{dB}
dB=\frac{\alpha}{\pi}\frac{da}{a} +\frac{\beta}{\pi}\frac{db}{b} +\frac{\gamma}{\pi}\frac{dc}{c}.
\end{equation}

We also have with $z=z_1+iz_2$
\begin{equation}\label{Bx}
\frac{\partial B(z)}{\partial z_1} = -\frac1{\pi}\Im\log(z)\Re\left(\frac{1}{1-z}\right)-\frac1{\pi}\Re\left(\frac1z\right)\Im\log(1-z),\end{equation}
and
\begin{equation}\label{By}
\frac{\partial B(z)}{\partial z_2} = \frac1{\pi}\Im\log(z)\Im\left(\frac{1}{1-z}\right)+\frac1{\pi}\Im\left(\frac{1}{z}\right)\Im\log(1-z).\end{equation}
Combining gives
\begin{equation}\label{Bz}
\frac{\partial B(z)}{\partial z} = -\frac1{2\pi}\frac{\arg z }{1-z}-\frac1{2\pi}\frac{\arg(1-z)}{z}.\end{equation}
Taking the $\overline z$ derivative then yields
\begin{equation}\label{LapB}\Delta B(z) = 4B(z)_{z\overline z} =
\Im\left(\frac2{\pi \overline z(1-z)}\right) = \frac{2\Im z}{\pi|z|^2|1-z|^2}.\end{equation}

\subsection{Arctic boundary}
\label{ap:arctic}

For $r<1$, to consider the behavior of $(x,y)$ along the arctic boundary of Theorem~\ref{xyH}, we consider $u=p+q i$ for $p\in\R$ and $q=\eps$ is small. The function $\A(z)$ consists of four analytic pieces giving rise to the following expressions. In all cases the expressions are the limits when $\eps\to0$ except in the range
$r^2<p<1$ where we give the leading asymptotics as $\eps\to0$.
$$R^2 \A(z) = \begin{cases}
\displaystyle -\frac{\pi r^2 p(1-p)}{1-r^2}&p<0\\[.15in]
\displaystyle  -\frac{\pi p(r^2-p)}{1-r^2}&0<p<r^2\\[.15in]
\displaystyle  -i \frac{r^2 \eps^2}{(1-p)(p-r^2)} + \frac{r^2(-2r^2+p+r^2p)\eps^3}{3p(p-r^2)^2(1-p)^2}+O(\eps^4)&r^2<p<1\\[.15in]
\displaystyle  \frac{\pi r^2 p(p-1)}{1-r^2}&1<p.
\end{cases}$$
and
$$r^2 \A(w)-R^2 \A(z) = \begin{cases}
\displaystyle \frac{-\pi  r^2 \left(p^2-2p+r^2\right)}{1-r^2}&p<0\\[.15in]
\displaystyle  -\frac{\pi(r^2-p)^2}{1-r^2}&0<p<r^2\\[.15in]
\displaystyle -\frac{2 r^2 \eps^3}{3p(p-r^2)(1-p)}+O(\eps^4)&r^2<p<1\\[.15in]
\displaystyle \frac{-\pi r^2(p-1)^2}{1-r^2}&1<p.\end{cases}$$

We also have 
$$(1-u)(u-r^2)\theta = \begin{cases}
\displaystyle -\pi(1-p)(r^2-p)&p<r^2\\[.15in] 
\displaystyle \left(1-r^2\right) \eps -\frac{i
   \left(r^2-1\right)
   \left(2 p-r^2-1\right)}{(p-1)
   \left(p-r^2\right)} \eps^2 \\[.15in] 
   \displaystyle \hphantom{(1-r^2)} +\frac{\left(r^2-1
   \right)\left(6 p^2-6 p
   \left(r^2+1\right)+r^4+4
   r^2+1\right)}{3 (p-1)^2
   \left(p-r^2\right)^2} \eps^3  +O\left(\eps^4\right)  &r^2<p<1\\[.15in]
\displaystyle -\pi(p-1)(p-r^2)&1<p.
\end{cases}$$

Let $\Re \F=F$. Then using the fact that $\F$ satisfies the Cauchy-Riemann equations, we have
$$k(p) =\frac{F_p(1+r^2-2p)+ F_{pp}(p-r^2)(1-p)}{r^2}.$$

\bibliographystyle{hmralphaabbrv}
\bibliography{fivevertex} 

\end{document}